\documentclass[10pt, a4paper, reqno, twoside]{amsart}

\numberwithin{equation}{section}

\usepackage{amsmath, amsfonts, amsthm, amssymb, mathtools, multicol, scalefnt, relsize, mathbbol, graphicx, enumitem, xcolor, slashed, xfrac, autobreak, lipsum, euflag}
\usepackage[mathscr]{euscript}
\usepackage[T1]{fontenc}
\usepackage[utf8]{inputenc}
\usepackage[all]{xy}
\setlist[itemize]{itemsep=0pt, parsep=0pt, partopsep=0pt, topsep=1pt, leftmargin=30pt}

\usepackage[hypertexnames=false,
    backref=page,
    pdftex,
    pdfpagemode=UseNone,
    breaklinks=true,
    extension=pdf,
    colorlinks=true,
    linkcolor=blue,
    citecolor=blue,
    urlcolor=blue,
]{hyperref}

\usepackage{cleveref}

\theoremstyle{plain}
\newtheorem{theorem}{Theorem}[section]
\newtheorem{lemma}[theorem]{Lemma}
\newtheorem{corollary}[theorem]{Corollary}
\newtheorem{proposition}[theorem]{Proposition}

\theoremstyle{definition}
\newtheorem{definition}[theorem]{Definition}

\newtheorem{question}[theorem]{Question}
\newtheorem{remark}[theorem]{Remark}
\newtheorem{example}[theorem]{Example}

\theoremstyle{plain}

\DeclareSymbolFontAlphabet{\mathbb}{AMSb}
\DeclareSymbolFontAlphabet{\mathbbl}{bbold}

\makeatletter
\@namedef{subjclassname@2020}{\textup{2020} Mathematics Subject Classification}
\makeatother

\allowdisplaybreaks[4]

\title{Yamabe-type problems on compact Hermitian manifolds}

\author[D.\ Angella]{Daniele Angella}
\address[Daniele Angella]{
Dipartimento di Matematica e Informatica\\
Universit\`a degli Studi di Firenze\\
viale Morgagni 67/a\\
50134 Firenze, Italy
}
\email{daniele.angella@unifi.it}

\author[F.\ Pediconi]{Francesco Pediconi}
\address[Francesco Pediconi]{
Dipartimento di Scienze Matematiche ``Giuseppe Luigi Lagrange''\\
Politecnico di Torino\\
corso Duca degli Abruzzi 24\\
10129 Torino, Italy
}
\email{francesco.pediconi@polito.it}

\author[C.\ Scarpa]{Carlo Scarpa}
\address[Carlo Scarpa]{
Institut Camille Jordan\\
Universit\'e Claude Bernard Lyon 1\\
43 Boulevard du 11 novembre 1918\\
69622 Villeurbanne Cedex, France
}
\email{scarpa@math.univ-lyon1.fr}

\author[C.\ Spotti]{Cristiano Spotti}
\address[Cristiano Spotti]{
Department of Mathematics\\
Aarhus University\\
Ny Munkegade 118\\
8000 Aarhus C, Denmark
}
\email{c.spotti@math.au.dk}

\author[O.\ J.\ Windare]{Oluwagbenga Joshua Windare}
\address[Oluwagbenga Joshua Windare]{
Dipartimento di Matematica e Informatica ``Ulisse Dini''\\
Universit\`a degli Studi di Firenze\\
viale Morgagni 67/a\\
50134 Firenze, Italy
}
\email{oluwagbengajoshua.windare@unifi.it}

\thanks{
The first-named author was partially supported by project PRIN2022 ``Real and Complex Manifolds: Geometry and Holomorphic Dynamics'' (code 2022AP8HZ9) and by GNSAGA of INdAM. The second-named author was partially supported by GNSAGA of INdAM. The third-named author is supported by a MSCA Postdoctoral Fellowship, funded by the Research and Innovation framework programme Horizon Europe {\normalsize\euflag} under grant agreement n.101149320. The fourth-named author is partially supported by the Villum YIP+ program 00053062. The last-named author was partially supported by Paolo de Bartolomeis' research grant ``Complex and Symplectic geometry" and by GNSAGA of INdAM.
}

\begin{document}

\begin{abstract}
We introduce and study a one-parameter Hermitian deformation of the Yamabe problem on compact complex manifolds. The deformation is defined by adding natural torsion terms to the Riemannian scalar curvature, and includes both the classical Yamabe equation and a scalar-curvature equation arising in locally conformally K\"ahler geometry as a momentum map. We analyze criteria for the existence of solutions, and discuss several examples.
\end{abstract}

\maketitle

\section{Introduction}
\setcounter{equation} 0

The existence of Riemannian metrics with constant scalar curvature has long been a central problem in differential geometry. The Yamabe problem \cite{MR0125546}, whose celebrated solution is due to Trudinger \cite{MR240748}, Aubin \cite{MR0431287}, and Schoen \cite{MR788292}, shows that within the conformal class of any Riemannian metric on a compact manifold, one can always find another metric of constant scalar curvature. For a unified expository account, see \cite{MR0888880}.

In the context of Hermitian metrics on compact complex manifolds, the situation is complicated by the fact that one may consider \emph{different notions} of curvature. Indeed, one could argue that in this context it is also natural to consider the \emph{Chern connection} of the manifold rather than the Levi-Civita one. Thus, one is therefore led to consider a Yamabe-type problem for the Chern connection \cite{MR3696598}: \emph{given a compact Hermitian manifold $(M^{2n},J,\omega)$, is there a function $f$ such that the rescaled metric $e^f\omega$ has constant Chern scalar curvature?} Despite being formally similar to the Riemannian Yamabe problem, this Chern--Yamabe problem does not share many of its good analytic properties. Most importantly, it is not, in general, variational \cite[Proposition 5.3]{MR3696598}. Nonetheless, several important results have been obtained for this problem; see, for example \cite{MR4070351, MR4184828, MR4173158, MR3833814, MR4203643, MR4236644, MR4554058, MR4411750, MR4571872, MR4971815, MR4418334, MR4471719}.

In this paper, we establish several existence results for a variant of the Chern--Yamabe problem, which modifies the Chern scalar curvature by a term depending on the torsion of the Chern connection of the Hermitian metric. Our starting point was the observation that the expression
\begin{equation}\label{eq:momentum_intro}
    \mu(\omega)\coloneqq \operatorname{scal}^{\mathrm{Ch}}(\omega) + \frac{n}{n+1} d^*\theta,
\end{equation}
which is associated with a Donaldson--Fujiki-type infinite-dimensional symplectic reduction in the locally conformally K\"ahler case \cite{MR4960067}, gives rise to a variational problem in the conformal class (see also \cite{MR4382665} for similar considerations in the globally conformally K\"ahler case). Here $\theta$ is the Lee form of $\omega$, that is, the unique $1$-form such that $ d\omega^{n-1} = \theta \wedge \omega^{n-1}$.
Furthermore, we noticed that there exists a one-parameter family, natural from the Hermitian point of view, of variational problems that connects the above expression to the classical Yamabe problem (see Remark \ref{rmk:psi-a-b}).
More precisely, for $t \geq 0$, we consider the \emph{$t$-deformed scalar curvature}
$$
\mu^{t}(\omega) \coloneqq \mathrm{scal}^{\mathrm{Ch}}(\omega) + \frac{t-n}{n-1}\, d^{*}\theta  + \frac{t-2n}{4}\,|T|^2,
$$
where $T$ is the torsion of the Chern connection.
Indeed, by a classical formula \cite{MR0742896}, for $t=2n-1$ we have that $\mu^{2n-1}(\omega)$ coincides exactly with the Riemannian scalar curvature of the underlying metric $g$, while for $t=2n$ we recover \eqref{eq:momentum_intro}.

Thus, in this note we are interested in  the following question:
\begin{quote}
    \textit{Given a compact complex manifold $(M,J)$ and a Hermitian metric $\omega$ on it, does there exist metric in the conformal class of $\omega$ having constant $\mu^t$?}
\end{quote}

In particular, as noted above, the special case $t=2n$ appears to be both of analytic and geometric interest. It is natural to wonder if the connection between Yamabe problems and stability phenomena of \cite{LLS23} also manifests in this setting.

An important remark is that the sign of the deformed scalar curvature for $t=0$ plays a special role in our study of the Yamabe problems for other values of $t$, see Theorem \ref{thm:aubins_crit}. Moreover, in the locally conformally K\"ahler case, $\mu^0$ coincides with the scalar curvature of the Weyl connection, see Remark \ref{rmk:weyl}.

\subsection{Summary of results} \hfill \par

In Lemma \ref{lem:confchange}, we compute the conformal variation of the quantities $\mu^t(\omega)$, which allows us to reformulate the $t$-deformed Yamabe-type  problem as the following elliptic PDE for a positive, smooth function $u$:
\begin{equation}\label{eq:yamabe_intro}
    \frac{2t}{n-1}\Delta u + \mu^t(\omega)\,u = c\, u^{\frac{n+1}{n-1}}. 
\end{equation}
Note that the exponent
$$ \frac{n+1}{n-1}=2^*-1 , \qquad \text{ where } 2^* = \frac{2n}{n-1} $$
is the critical exponent in the Sobolev embedding on $M$. In Theorem \ref{thm:variational}, we show that \eqref{eq:yamabe_intro} is the Euler--Lagrange equation of the \emph{$t$-deformed Einstein--Hilbert functional}
\begin{equation}\label{eq:EH-t_intro}
    \mathrm{EH}^{t}(\omega) \coloneqq \frac{\int_M\mu^{t}(\omega)\frac{\omega^n}{n!}}{\left(\int_M\frac{\omega^n}{n!}\right)^{\frac{n-1}{n}}},
\end{equation}
restricted to any conformal class of Hermitian metrics on $M$.

The critical equation \eqref{eq:yamabe_intro} is closely related to the Yamabe equation. This analogy makes it possible, in particular, to show that the sign of the expected constant $t$-deformed scalar curvature (\emph{i.e.}, the constant $c$ in \eqref{eq:yamabe_intro}) is a conformal invariant, and we establish uniqueness of solutions up to homotheties when the expected value is non-positive (see Proposition \ref{prop:rigidity}).

As the deformed Einstein--Hilbert functional is bounded from below on each conformal class (see Proposition \ref{prop:minimizer}), the \emph{$t$-deformed Yamabe-type invariant}
\begin{equation}
\lambda_{(M,J)}^{t}(\{\omega\}) \coloneqq \inf \big\{ \mathrm{EH}^t(\omega') : \omega'\in\{\omega\} \big\}
\end{equation}
is well-defined. As happens for the classical Yamabe invariant, $\lambda^t$ is always bounded from above. Moreover, when it is strictly less than its critical value, then there exists a positive solution to the deformed Yamabe-type equation, as it follows by applying Aubin's result \cite{MR0431287}.

\begin{theorem}[{see \cite{MR0431287, Hebey_ICTP}}] \label{thm:existence_subcritical_intro}
Let $(M^{2n},J,\omega)$ be a compact Hermitian manifold of complex dimension $n$. Let $t>0$. Then
\begin{equation}\label{eq:bound}
\lambda^{t}_{(M,J)}(\{\omega\}) \leq 2n t\, \mathrm{Vol}(S^{2n})^{\frac{1}{n}} \,\, .
\end{equation}
Moreover, if $\lambda^{t}_{(M,J)}(\{\omega\}) < 2n t\, \mathrm{Vol}(S^{2n})^{\frac{1}{n}}$, then there exists a positive smooth solution to \eqref{eq:yamabe_intro}, which is a minimiser of the $t$-deformed Einstein--Hilbert functional $\mathrm{EH}^t$ restricted to the conformal class $\{\omega\}$ as in \eqref{eq:EH-t_intro}.
\end{theorem}

\begin{remark}
It is still not clear whether equality in \eqref{eq:bound} can ever be realised, beside the case of $\mathbb CP^1$.
\end{remark}

In Section \ref{sec:invariant}, we show several results linking $\lambda^{t}_{(M,J)}(\{\omega\})$ with the Gauduchon degree $\Gamma_{(M,J)}(\{\omega\})$ of the conformal class $\{\omega\}$. In particular, for $0<t\leq 2n$, Theorem~\ref{thm:existence_subcritical_intro} applies to every conformal classes with non-positive Gauduchon degree.

\begin{theorem}
\label{thm:existence-negative-intro}
Let $(M^{2n},J,\omega)$ be a compact Hermitian manifold of complex dimension $n$. Fix $0<t\leq 2n$. If $\Gamma_{(M,J)}(\{\omega\}) < 2n t\, \mathrm{Vol}(S^{2n})^{\frac{1}{n}}$, then there exists a constant $t$-deformed scalar curvature metric in the conformal class $\{\omega\}$. Moreover, if $\Gamma_{(M,J)}(\{\omega\}) \leq 0$, then such a metric is unique up to scaling, and in particular it is non-positive.

This includes the case of compact complex manifolds with non-negative Kodaira dimension.
\end{theorem}

Further results are obtained for special conformal classes, more precisely, if the metrics are locally conformally K\"ahler (see Corollary \ref{cor:existence}).
This applies in particular to prove that, if the conformal class contains a K\"ahler metric, then it also contains a constant momentum Hermitian metric, whatever the sign of the deformed Yamabe invariant is. This overcomes the limitation of the negative curvature in the corresponding Chern--Yamabe-type problems.
We notice that these results are based on a criterion for the existence of metrics with constant $t$-deformed scalar curvature that is more general than Theorem \ref{thm:existence-negative-intro}, drawing on Aubin's work \cite{MR0431287}. See in particular Theorem \ref{thm:aubins_crit}. Moreover, we will use this criterion to exhibit a large class of examples of manifolds admitting solutions of \eqref{eq:yamabe_intro} in Section \ref{sec:examples}, see Theorem \ref{thm:prod-kahler}.

Furthermore, by exploiting the computations for cohomogeneity-one Hermitian non-K\"ahler manifolds developed in \cite{MR4567558}, we construct the first non-homogeneous example of a Hermitian non-K\"ahler metric with \emph{positive} constant momentum, see Theorem \ref{thm:ex}.

\subsection{Further considerations} \hfill \par

There are several natural directions not investigated in this paper. For example, one might want to consider the critical points of the $t$-deformed Einstein--Hilbert functional on the full space of Hermitian metrics on $(M, J)$, rather than the restriction to a given conformal class only. This is related to considering the \textit{$t$-deformed Yamabe invariant} 
$$
\sigma^{\mathbb{C}}_t(M,J) \coloneqq \sup_{\omega\text{ Hermitian}} \lambda_{(M,J)}^{t}(\{\omega\}),
$$
which is a \emph{function} of the parameter $t$. Note that for $t=2n-1$ we do not re-obtain the classical Yamabe invariant, as we are considering only conformal classes of Hermitian metrics on $M$, rather than the full set of Riemannian metrics.
Furthermore, it seems interesting to investigate further relations between this invariant and the underlying complex geometry, as in \cite{MR1674105, MR4785592}.
We leave these considerations to future work.

We remark that one could further restrict this invariant by restricting attention to classes of special Hermitian metrics, such as locally conformally K\"ahler, K\"ahler, balanced, etc. A similar point of view has been taken up in the recent work \cite{LLS25}, in which a \emph{Sasakian} version of such an invariant has been shown to be related to the K-semistability of K\"ahler cones.

\smallskip

Several variations of the Yamabe problem have been considered for Hermitian manifolds. A notable example is \cite{MR1967044}, which uses the same results of Aubin's to establish a very general existence result for a $J$-twisted Yamabe problem. Note, however, that the momentum \eqref{eq:momentum_intro} is quite different from the \emph{$J$-twisted scalar curvature} considered in \cite{MR1967044}.
Furthermore, other variants of the Chern--Yamabe problem have been considered; see, for example, \cite{MR4554058} and \cite{MR4184828}, where one considers the scalar curvature of a different connection than the Chern one, or the \emph{almost complex} setting. Given the non-variational nature of the Chern--Yamabe equation, the existence results in these cases are limited to the negative curvature cases. It would be natural to adapt our equations in similar directions, considering possibly a different family of equations than the $t$-deformed scalar curvature equation.

\subsection{Organisation of the paper} \hfill \par

In Section \ref{sec:preliminaries}, we collect some definitions and 
a few computations that will be crucial for the study of the $t$-deformed scalar curvatures. In Section \ref{sec:yamabe}, we prove the variational characterisation of our equations through the $t$-deformed Einstein--Hilbert functional, and we establish a uniqueness result in the non-positively curved case.
Section \ref{sec:invariant} is the technical heart of the paper. We review some classical results of Aubin and show what they imply in our situation to prove existence criteria for solutions of equation \eqref{eq:yamabe_intro}, {\itshape e.g.} Theorem \ref{thm:existence-negative-intro}, Theorem \ref{thm:aubins_crit} and its application Corollary \ref{cor:existence}.
In Section \ref{sec:examples}, we first show a general existence result for some conformal classes on product Hermitian manifolds $X \times Y$ such that $X$ is a K\"ahler manifold. Then, we show an explicit example of a ruled manifold with positive constant momentum.
Finally, the Appendix collects several computations, and a proof of Aubin's result stated in Theorem \ref{thm:aubins_crit}.

\bigskip

\noindent \textbf{Acknowledgements.} We would like to thank Nicolina Istrati for useful exchanges and comments on an earlier version of this paper.

\section{Preliminaries and Notation} \label{sec:preliminaries}

Let $(M^{2n}, J, g)$ be a compact Hermitian manifold of complex dimension $n$, with associated $(1,1)$-form $\omega \coloneqq g(J\cdot, \cdot)$. Once $J$ is fixed, we will often use $\omega$ and $g$ interchangeably, referring implicitly to the corresponding metric whenever the meaning is clear from the context. Since
$
\omega^n = n!\,d\mathrm{vol}_g
$,
where $d\mathrm{vol}_g$ denotes the volume form induced by $g$, for every smooth function $f: M \to \mathbb{R}$ and for every $1 \leq p < \infty$, the $L^p$ norm of $f$ is given by
$$
\|f\|_{L^p} = \left(\int_M |f|^p\frac{\omega^n}{n!}\right)^{\frac{1}{p}}.
$$

We denote by $\Delta \coloneqq d^* d + d d^*$ the Hodge-de Rham Laplacian with respect to $\omega$, where $d^*$ is the adjoint of $d$ with respect to the global $L^2$-inner product induced by $g$. We recall that, for any smooth real-valued function $f$ on $M$, at a local maximum point $p$ one has $(\Delta f)(p) \geq 0$. We also denote by $\nabla$ the gradient operator on smooth functions induced by $g$, \emph{i.e.}, $\nabla f = (df)^{\sharp}$, where $\sharp$ denotes the musical isomorphism induced by $g$.
\smallskip

We denote by $D^g$ the Levi-Civita connection of $g$ and by $D^{\mathrm{Ch}}$ the \emph{Chern connection} of $(J,g)$, which are related by
$$
g(D^{\mathrm{Ch}}_XY,Z) \coloneqq g(D^g_XY,Z) -\tfrac{1}{2}d\omega(JX,Y,Z) \,\, .
$$
We recall that $D^{\mathrm{Ch}}$ is the unique linear connection that is Hermitian, \emph{i.e.}, $D^{\mathrm{Ch}}J = D^{\mathrm{Ch}}g = 0$, and whose torsion $T$, given by
\begin{equation} \label{eq:T}
-2g(T(X,Y),Z) = d\omega(JX,Y,Z) +d\omega(X,JY,Z) \,\, ,
\end{equation}
has vanishing $(1,1)$-part, namely
$$
JT(X,Y) = T(JX,Y) = T(X,JY) \,\, ,
$$
see, \emph{e.g.}, \cite{MR1456265}. We denote by $\theta$ the \emph{Lee form}, defined as the trace of the torsion $T$. Explicitly, given a $g$-orthonormal local frame $\{ e_{\alpha} \}_{\alpha=1,\ldots, 2n}$ of $TM$, we have
\begin{equation} \label{eq:theta}
\theta \coloneqq \sum_{\alpha=1}^{2n} g( T(\_, e_{\alpha}), e_{\alpha}) \,\, .
\end{equation}
We also recall that the Lee form is characterized by the property
$$
{d}\omega^{n-1} = \theta \wedge \omega^{n-1} \,\, ,
$$
see, \emph{e.g.}, \cite[Th{\'e}or{\`e}me 3]{MR0470920} or \cite[p.\ 500]{MR0742896}.

\smallskip

There are several notable classes of special Hermitian metrics. Among these, \emph{balanced metrics} are characterized by the vanishing of the Lee form, \emph{i.e.}, $\theta = 0$. On the other hand, \emph{Gauduchon metrics} are characterized by the Lee form being co-closed, \emph{i.e.}, $d^*\theta = 0$. By \cite[Th\'eor\`eme 1]{MR0470920}, any conformal class of Hermitian metrics on a compact complex manifold contains a unique Gauduchon metric of volume $1$. Finally, a Hermitian metric is called \emph{locally conformally K{\"a}hler} if
$$
d\omega = \frac{1}{n-1}\theta \wedge \omega \,\, , \quad d\theta = 0 \,\, ,
$$
and \emph{Vaisman} if it is locally conformally K{\"a}hler and $D^g\theta = 0$ (see, \emph{e.g.}, \cite{MR4771164}).
\smallskip

We denote by
\begin{equation} \label{eq:Theta}
\Theta(X,Y,Z,W) \coloneqq g(D^{\mathrm{Ch}}_XD^{\mathrm{Ch}}_YZ -D^{\mathrm{Ch}}_YD^{\mathrm{Ch}}_XZ -D^{\mathrm{Ch}}_{[X,Y]}Z,W)
\end{equation}
the curvature tensor of the Chern connection, and by $\mathrm{scal}^{\mathrm{Ch}}(\omega)$ the \emph{Chern scalar curvature}, obtained by a trace of \eqref{eq:Theta}. Precisely, if $\{ e_i, Je_i \}_{i=1,\ldots, n}$ is a $(J,g)$-unitary local frame of $TM$, then
\begin{equation} \label{eq:scalCh}
\mathrm{scal}^{\mathrm{Ch}}(\omega) \coloneqq -2\sum_{i,j=1}^{n} \Theta(e_i,Je_i,e_j,Je_j) \,\, .
\end{equation}
With this normalization, \cite[Eq.\ (33)]{MR0742896} implies that the usual \emph{Riemannian scalar curvature} $\mathrm{scal}(g)$ is related to the Chern scalar curvature by
\begin{equation} \label{eq:scal-scalCh}
\mathrm{scal}(g) = \mathrm{scal}^{\mathrm{Ch}}(\omega) +d^*\theta -\tfrac14|T|^2 \,\, .
\end{equation}
In particular, $\mathrm{scal}(g) = \mathrm{scal}^{\mathrm{Ch}}(\omega)$ on a K\"ahler manifold.
\smallskip

Finally, we collect here some formulas for the conformal change of geometric quantities associated with a Hermitian metric, which can be obtained by straightforward computations. For the sake of completeness, we include the proof of the following lemma in the Appendix.

\begin{lemma} \label{lem:confchange}
Let $(M^{2n},J,\omega)$ be a compact Hermitian manifold of complex dimension $n$. Let $f\colon M \to \mathbb{R}$ be a smooth function, and set $\omega' \coloneqq e^{f} \omega$. Denote by $T'$, $\theta'$, and $d^*_{g'}$ the torsion, the Lee form, and the codifferential associated with $\omega'$, respectively. Then
\begin{align}
\mathrm{scal}^{\mathrm{Ch}}(\omega') &= e^{-f} \Big(\mathrm{scal}^{\mathrm{Ch}}(\omega) + n\,\Delta f + n\,g(\nabla f, \theta^{\sharp}) \Big) \,\, , \label{eq:confscalCh} \\
d^*_{g'}\theta' &= e^{-f} \Big( d^*\theta+(n-1)\Delta f -(n-1) g( \nabla f, \theta^{\sharp}) - (n-1)^2| \nabla  f|^2 \Big) \,\, , \label{eq:confcodiftheta}\\
|T'|_{g'}^2 &= e^{-f}\Big(|T|^2 +4g( \nabla f, \theta^{\sharp}) +2(n-1)\lvert \nabla f\rvert^2\Big) \label{eq:confT} \,\, ,
\end{align}
where all the quantities on the right-hand sides are computed with respect to $\omega$.
\end{lemma}

\section{A Yamabe problem for the deformed scalar curvatures}
\label{sec:yamabe}
\setcounter{equation} 0

\subsection{The deformed scalar curvatures} \hfill \par

We introduce the following family of \emph{deformed scalar curvatures} on a given compact Hermitian manifold, defined by perturbing the Chern scalar curvature with additional torsion terms.

\begin{definition}\label{def:mu-t}
Let $(M^{2n},J)$ be a compact complex manifold of complex dimension $n$. For every Hermitian metric $\omega$ on $(M^{2n},J)$ and for every $t \in \mathbb{R}$, we define the \emph{$t$-deformed scalar curvature of $\omega$} by
\begin{equation}\label{eq:mu_t}
\mu^{t}(\omega) \coloneqq \mathrm{scal}^{\mathrm{Ch}}(\omega) +\frac{t-n}{n-1}d^*\theta + \frac{t-2n}{4}|T|^2 \,\, ,
\end{equation}
where $\mathrm{scal}^{\mathrm{Ch}}(\omega)$ is the Chern scalar curvature of $\omega$, defined in \eqref{eq:scalCh}, $T$ is the torsion of the Chern connection of $\omega$, defined in \eqref{eq:T}, and $\theta$ is the Lee form of $\omega$, defined in \eqref{eq:theta}.
\end{definition}

We remark that, for $t=2n$, the $2n$-deformed scalar curvature admits a geometric interpretation, in the locally conformally K\"ahler setting, as the momentum map considered in \cite{MR4960067}, albeit here forgetting a condition on the dual of the Lee form. For this reason, we define the \emph{momentum of $\omega$} to be its $2n$-deformed scalar curvature, namely
\begin{equation} \label{eq:moment}
\mu^{2n}(\omega) = \mathrm{scal}^{\mathrm{Ch}}(\omega) +\frac{n}{n-1}d^*\theta \,\, .
\end{equation}
We also remark that \eqref{eq:scal-scalCh} implies that the $(2n{-}1)$-deformed scalar curvature coincides with the Riemannian scalar curvature, namely
$$
\mu^{2n-1}(\omega) = \mathrm{scal}^{\mathrm{Ch}}(\omega) +d^*\theta -\tfrac14|T|^2 = \mathrm{scal}(g_{\omega}) \,\, .
$$

We are interested in \emph{Yamabe problem for the $t$-deformed scalar curvature}, that is, in finding Hermitian metrics with constant $t$-deformed scalar curvature within a fixed conformal class. For this reason, we prove the following lemma.

\begin{proposition} \label{prop:linear}
Let $(M^{2n},J,\omega)$ be a compact Hermitian manifold of complex dimension $n$, let $f: M \to \mathbb{R}$ be a smooth function, and let $c \in \mathbb{R}$ be a constant. Then the conformal metric $e^f\omega$ has constant $t$-deformed scalar curvature equal to $c$, namely
$$
\mu^{t}(e^f\omega) =c \,\, ,
$$
if and only if the function $u \coloneqq e^{\frac{n-1}{2}f}$ is a smooth positive solution of the semilinear equation
\begin{equation} \label{eq:t-DY} \tag{$t$-DY}
    \frac{2t}{n-1}\Delta u + \mu^t(\omega)\,u = c\, u^{\frac{n+1}{n-1}} \,\, .
\end{equation}
\end{proposition}

\begin{proof}
By Lemma \ref{lem:confchange} and Definition \ref{def:mu-t}, we have
\begin{equation} \label{eq:mu_t_conformal}
e^{f}\,\mu^t(\omega') = \mu^t(\omega) + t\Delta f -\frac{t(n-1)}{2} |\nabla f|^2 \,\, .
\end{equation}
Standard computations show that, for $f = \frac{2}{n-1}\log u$, we have
$$
\Delta f = \frac{2}{n-1}\left(\frac{\Delta u}{u}+\frac{|\nabla u|^2}{u^2}\right) \,\, , \quad
|\nabla f|^2 = \frac{4}{(n-1)^2}\frac{|\nabla u|^2}{u^2} \,\, .
$$
Substituting these identities into \eqref{eq:mu_t_conformal}, we obtain
$$
u^{\frac{2}{n-1}} \mu^t\left(u^{\frac{2}{n-1}}\omega\right)
=
\mu^t(\omega) + \frac{2t}{n-1}\frac{\Delta u}{u}
$$
and so it follows that $\mu^t(e^f\omega) = c$ if and only if $u = e^{\frac{n-1}{2}f}$ is a smooth positive solution of \eqref{eq:t-DY}.
\end{proof}

\begin{remark}\label{rmk:psi-a-b}
Notice that in the proof of Proposition \ref{prop:linear}, the gradient terms $g(\nabla f,\theta^\sharp)$ appearing in Lemma \ref{lem:confchange} cancel out in \eqref{eq:mu_t_conformal}. Indeed, if one considers the geometric quantity
$$
\psi^{a,b}(\omega) \coloneqq \mathrm{scal}^{\mathrm{Ch}}(\omega) +a\,d^*\theta +b\,|T|^2 \,\, ,
$$
then a straightforward computation based on Lemma \ref{lem:confchange} gives
\begin{multline*}
e^{f}\psi^{a,b}(e^{f}\omega) = \psi^{a,b}(\omega) +(n +(n-1)a)\Delta f + (n -(n-1)a +4b)g(\nabla f,\theta^\sharp) \\
-(n-1)((n-1)a +2b)|\nabla f|^2 \,\, .
\end{multline*}
The coefficients in \eqref{eq:mu_t} were chosen so that
$$
(n-1)a -4b = n
$$
because this condition leads to a variational formulation of \eqref{eq:t-DY} (see Theorem \ref{thm:variational}).
\end{remark}

\begin{remark} \label{rem:crit}
Notice that the exponent $\frac{n+1}{n-1}$ appearing in \eqref{eq:t-DY} can be rewritten as $q-1$, where
$$
q \coloneqq 2^* = \frac{2n}{n-1}
$$
is the critical exponent for the Sobolev embedding $W^{1,2}(\mathbb R^{2n})\hookrightarrow L^p(\mathbb R^{2n})$, see, \emph{e.g.}, \cite{MR0448404,MR0463908} or \cite[Theorem 3.3]{MR0888880}. Therefore, \eqref{eq:t-DY} is a critical elliptic PDE on a compact Riemannian manifold, and belongs to the class of equations considered by Aubin \cite[Chapter 5]{MR1636569} and Hebey--Vaugon \cite{MR1854089}, among others. Moreover, for $t=2n-1$, it coincides with the classical Yamabe equation
$$
\frac{2(2n-1)}{n-1}\Delta u+\mathrm{scal}(g)\, u = c\,u^{q-1} \,\, ,
$$
see, \emph{e.g.}, \cite[p.\ 38]{MR0888880}.
\end{remark}

\begin{remark}\label{rmk:weyl}
Note that a peculiar phenomenon occurs when $t=0$. Indeed, equation \eqref{eq:t-DY}
degenerates and reduces to
$$
c\, u^{\frac{2}{n-1}} = \mu^0(\omega) \,\, ,
$$
which admits a positive solution whenever $\mu^{0}(\omega)$ has a sign that agrees with the sign of $c$. For this reason, we consider \eqref{eq:t-DY} only for $t>0$. Nonetheless, the sign of the $0$-deformed scalar curvature $\mu^0$ plays a role in the existence problem for constant $t$-deformed scalar curvature metrics (see
Theorem \ref{thm:aubins_crit}). Moreover, if the metric $\omega$ is locally conformally K{\"a}hler, then $\mu^0(\omega)$ coincides with the scalar curvature of the \emph{Weyl connection} of $\omega$.
\end{remark}

Inspired by the classical Yamabe problem \cite{MR0125546} and the Chern--Yamabe problem \cite{MR3696598}, we propose the following question.

\begin{question}[Deformed scalar curvature Yamabe-type problem]
Let $(M^{2n}, J,\omega)$ be a compact Hermitian manifold of complex dimension $n$. Fix any $t>0$. Does there exist a positive smooth solution to \eqref{eq:t-DY}?
In other words, does the conformal class $\{\omega\}$ contain a metric with constant $t$-deformed scalar curvature?
\end{question}

As a consequence of Remark \ref{rem:crit}, many arguments used for the Yamabe equation can be adapted to our setting. For example, the following rigidity property holds.

\begin{proposition} \label{prop:rigidity}
Let $(M^{2n}, J)$ be a compact connected complex manifold, and fix $t>0$. Let also $\omega$ and $\omega'$ be two Hermitian metrics with constant $t$-deformed curvature that are conformal to each other. Then, either $\mu^{t}(\omega) = \mu^{t}(\omega') = 0$, or else $\mu^{t}(\omega)$ and $\mu^{t}(\omega')$ are both nonzero and have the same sign. Moreover, if $\mu^{t}(\omega) \leq 0$, then $\omega'$ is a constant rescaling of $\omega$.
\end{proposition}

\begin{proof}
By assumption, there exists a positive smooth function $u$ on $M$ such that $\omega' = u^{\frac{2}{n-1}}\omega$. Let $c \coloneqq \mu^{t}(\omega)$ and $c' \coloneqq \mu^{t}(\omega')$. Then, by \eqref{eq:t-DY}, it follows that
\begin{equation} \label{eq:ll'}
\frac{2t}{n-1}\Delta u +c u = c' u^{q-1} \,\, ,
\end{equation}
where $\Delta$ is the Laplacian associated with $\omega$. Integrating \eqref{eq:ll'} over $M$ with respect to the volume form of $\omega$, we find
$$
c \int_M u\,\frac{\omega^n}{n!} = c' \int_M u^{q-1} \frac{\omega^n}{n!}  .
$$
Since $u>0$, it follows that either $c = c' = 0$, or they are both nonzero with the same sign.

Assume now that $c <0$. Then, up to rescaling both the metrics $\omega$ and $\omega'$ by positive constants, we can assume that $c = c' = -1$, and so \eqref{eq:ll'} reads as
\begin{equation} \label{eq:-1-1}
\frac{2t}{n-1}\Delta u -u = -u^{q-1} \,\, .
\end{equation}
Evaluating \eqref{eq:-1-1} at a minimum point $x_{\mathrm{min}} \in M$ of $u$, we get
$$
u(x_{\mathrm{min}})^{q-1} = -\frac{2t}{n-1}\Delta u(x_{\mathrm{min}}) +u(x_{\mathrm{min}}) \geq u(x_{\mathrm{min}})  .
$$
Since $q>2$, it follows that $u(x_{\mathrm{min}}) \geq 1$. Evaluating \eqref{eq:-1-1} at a maximum point $x_{\mathrm{max}} \in M$ of $u$, we get
$$
u(x_{\mathrm{max}})^{q-1} = -\frac{2t}{n-1}\Delta u(x_{\mathrm{max}}) +u(x_{\mathrm{max}}) \leq u(x_{\mathrm{max}}) \,\, ,
$$
which in turn implies that $u(x_{\mathrm{max}}) \leq 1$. Hence $u \equiv 1$, and so $\omega' = \omega$.

Finally, if $c = 0$, then \eqref{eq:ll'} reads as $\Delta u = 0$, and so $u$ is constant.
\end{proof}

Let us recall that, given a compact Hermitian manifold $(M^{2n},J,\omega)$, the \emph{Gauduchon degree} \cite{MR0633563} is the invariant of the conformal class $\{\omega\}$ defined by
\begin{equation} \label{eq:Gdeg}
\Gamma_{(M,J)}(\{\omega\}) \coloneqq \int_M \mathrm{scal}^{\mathrm{Ch}}(\eta)\frac{\eta^n}{n!} \,\, ,
\end{equation}
where $\eta \in \{\omega\}$ is the unique Gauduchon metric of volume $1$ in the conformal class of $\omega$. Note that it coincides, up to a positive constant, with the degree of the anti-canonical line bundle of $(M,J)$ with respect to $\eta$ (see, \emph{e.g.}, \cite{MR0633563, MR3696598}). The Gauduchon degree provides a simple obstruction to the existence of metrics with positive constant $t$-deformed scalar curvature when $0 < t \leq 2n$. Indeed, the following result holds.

\begin{proposition} \label{prop:Gdeg}
Let $(M^{2n},J,\omega)$ be a compact Hermitian manifold of complex dimension $n$. Let $0 < t \leq 2n$ and let $c \in \mathbb{R}$ be a constant. Assume that there exists $\omega' \in \{\omega\}$ such that $\mu^t(\omega') = c$. If $\Gamma_{(M,J)}(\{\omega\})  \leq 0$, then $c \leq 0$. Moreover,
\begin{equation} \label{eq:cneg}
c\,\mathrm{Vol}(M,g_{\omega'})^{\frac{1}{n}}\leq \Gamma_{(M,J)}(\{\omega\}) \quad \text{if $c\leq 0$}\,\, .
\end{equation}
\end{proposition}

\begin{proof}
Let $\eta \in \{\omega\}$ the unique Gauduchon metric of volume $1$ in the conformal class of $\omega$ and let $f: M \to \mathbb{R}$ be a smooth function such that $\mu^{t}(e^f\eta) = c$. Since $\eta$ is Gauduchon and $0 < t \leq 2n$, it follows that
$$
\mu^{t}(\eta) = \mathrm{scal}^{\mathrm{Ch}}(\eta) +\frac{t-2n}{4}|T|^2 \leq \mathrm{scal}^{\mathrm{Ch}}(\eta) \,\, .
$$
Moreover, by \eqref{eq:mu_t_conformal}, we get
$$
c\,e^{f} = \mu^t(\eta) + t\Delta f -\frac{t(n-1)}{2} |\nabla f|^2
$$
and so, integrating over $M$ with respect
to the volume form of $\eta$, by \eqref{eq:Gdeg} we obtain
$$
c\int_M e^{f}\frac{\eta^n}{n!} = \Gamma_{(M,J)}(\{\omega\}) -\frac{t(n-1)}{2}\|\nabla f\|_{L^2}^2 \,\, .
$$
Therefore, if $\Gamma_{(M,J)}(\{\omega\}) \leq 0$, it follows that $c \leq 0$. Furthermore, since $\eta$ has unit volume, the H{\"o}lder inequality implies that
$$
\int_M e^{f}\frac{\eta^n}{n!} \leq \left(\int_M e^{nf}\frac{\eta^n}{n!}\right)^{\frac{1}{n}} = \mathrm{Vol}(M,g_{\omega'})^{\frac{1}{n}} \,\, ,
$$
and so \eqref{eq:cneg} holds.
\end{proof}

\begin{remark} \label{rmk:Gauduchon_deg}
As a consequence of Proposition \ref{prop:Gdeg}, the expected volume-normalized constant $t$-deformed scalar curvature is bounded from above by the Gauduchon degree on compact complex manifolds with non-negative Kodaira dimension when $0< t \leq 2n$ (see \cite{MR0633563}).
\end{remark}

\subsection{The deformed Yamabe problem as an Euler--Lagrange equation} \hfill \par

We now show that the constant $t$-deformed scalar curvature equation
$$
\mu^t(\omega) = c
$$
admits a variational interpretation as the Euler--Lagrange equation of a suitable functional. To this aim, we define the \emph{$t$-deformed Einstein--Hilbert functional}
\begin{equation} \label{eq:EH-t}
\mathrm{EH}^{t}(\omega) \coloneqq \frac{\int_M\mu^{t}(\omega)\frac{\omega^n}{n!}}{\left(\int_M\frac{\omega^n}{n!}\right)^{\frac{n-1}{n}}} \,\, ,
\end{equation}
where $\omega$ is any Hermitian metric on $(M^{2n},J)$. If we fix a reference metric $\omega$ and we restrict the $t$-deformed Einstein--Hilbert functional to its conformal class
$$
\{\omega\} \coloneqq \left\{ u^{\frac{2}{(n-1)}}\omega : u \in \mathcal{C}^{\infty}(M,\mathbb{R}) \, , \,\, u>0 \right\} \,\, ,
$$
then a direct computation shows that the map $\mathrm{EH}^{t}\big|_{\{\omega\}} : \{\omega\} \to \mathbb{R}$ takes the form
\begin{equation} \label{eq:EHt(u)}
\mathrm{EH}_{\omega}^{t}(u) \coloneqq
\mathrm{EH}^{t}\left(u^{\frac{2}{n-1}}\omega \right) = \|u\|^{-2}_{L^q}\int_M\left( u^2\mu^{t}(\omega) +\frac{2t}{n-1}|\nabla u|^2\right)\,\frac{\omega^n}{n!} \,\, .
\end{equation}
In the following, whenever no confusion can arise, we shall omit the reference metric $\omega$ from the notation $\mathrm{EH}_{\omega}^{t}(u)$. The role of the functional $\mathrm{EH}^{t}$ is clarified by the following result.

\begin{theorem} \label{thm:variational}
Let $(M^{2n},J,\omega)$ be a compact Hermitian manifold of complex dimension $n$. Then equation \eqref{eq:t-DY} is the Euler--Lagrange equation of the $t$-deformed Einstein--Hilbert functional $\mathrm{EH}^t$ defined in \eqref{eq:EH-t}, restricted to the conformal class of $\omega$.
\end{theorem}

\begin{proof}
For notational convenience, we set 
\begin{equation} \label{eq:mutot}
\boldsymbol{\mu}_{\mathrm{tot}}^t(u) \coloneqq \int_M\left( u^2\mu^{t}(\omega)+\frac{2t}{n-1}|\nabla u|^2\right)\,\frac{\omega^n}{n!} \,\, ,
\end{equation}
so that \eqref{eq:EHt(u)} takes the form
\begin{equation} \label{eq:EHt(u)2}
\mathrm{EH}^t(u) = \|u\|^{-2}_{L^q}\,\boldsymbol{\mu}_{\mathrm{tot}}^t(u) \,\, .
\end{equation}
Let $u_\varepsilon$ be a smooth path of smooth, real-valued, positive functions on $M$, with $u \coloneqq u_{\varepsilon=0}$ and $\dot u \coloneqq \frac{d}{d\varepsilon}|_{\varepsilon=0}u_\varepsilon$. A direct computation shows that
\begin{align*}
\frac{d}{d\varepsilon}\Big\vert_{\varepsilon=0} \|u_\varepsilon\|_{L^q} &=
\|u_\varepsilon\|_{L^q}^{1-q} \int_M u^{q-1}\,\dot{u}\,\frac{\omega^n}{n!} \,\, , \\
\frac{d}{d\varepsilon}\Big\vert_{\varepsilon=0} \boldsymbol{\mu}_{\mathrm{tot}}^t(u_\varepsilon) &=
2\int_M \left(\frac{2t}{n-1}\Delta u +\mu^t(\omega)\,u\right)\dot{u}\,\frac{\omega^n}{n!} \,\, .
\end{align*}
Hence, the first variation of the functional \eqref{eq:EHt(u)2} is computed as
\begin{align*}
\frac{d}{d\varepsilon}\Big\vert_{\varepsilon=0}\mathrm{EH}^t (u_\varepsilon) &= \|u\|_{L^q}^{-2}\left(\frac{d}{d\varepsilon}\Big\vert_{\varepsilon=0} \boldsymbol{\mu}_{\mathrm{tot}}^t(u_\varepsilon)\right)
-2\|u\|_{L^q}^{-3}\left(\frac{d}{d\varepsilon}\Big\vert_{\varepsilon=0} \|u_\varepsilon\|_{L^q}\right)\boldsymbol{\mu}_{\mathrm{tot}}^t(u) \\
&= \frac{2}{\|u\|_{L^q}^2}\int_M \left(\frac{2t}{n-1}\Delta u+\mu^t(\omega)\,u -\frac{\boldsymbol{\mu}_{\mathrm{tot}}^t(u)}{\|u\|_{L^q}^q}\,u^{q-1}\right)\dot{u}\,\frac{\omega^n}{n!} \,\, .
\end{align*}
Therefore, it follows that $\frac{d}{d\varepsilon}\big\vert_{\varepsilon=0}\mathrm{EH}^t (u_\varepsilon) = 0$ for all possible choices of $\dot{u}$ if and only if $u$ satisfies \eqref{eq:t-DY}.
\end{proof}

\begin{remark} \label{rem:lambda_value}
As a byproduct of the proof of Theorem \ref{thm:variational}, the expected value of $c$ in equation \eqref{eq:t-DY} is given by
$$
c = \frac{\boldsymbol{\mu}_{\mathrm{tot}}^t(u)}{\|u\|_{L^q}^q} \,\, ,
$$
where $\boldsymbol{\mu}_{\mathrm{tot}}^t(u)$ is defined in \eqref{eq:mutot}.
\end{remark}

The exponent $q$ in Equation \eqref{eq:t-DY} is of crucial importance (see Remark \ref{rem:crit}), and the variational problem at the critical exponent $q$ is substantially more delicate than its subcritical analogues. The following classical result illustrates this phenomenon.

\begin{proposition}[\cite{MR0125546}; see also \cite{MR0888880}, Proposition 4.2]\label{prop:minimizer}
Let $(M^{2n},J,\omega)$ be a compact Hermitian manifold of complex dimension $n$. For $2 \leq p \leq \infty$, consider the functional
$$ \mathrm{EH}^t_p(u) \coloneqq \frac{\boldsymbol{\mu}_{\mathrm{tot}}^t(u)}{\|u\|_{L^p}^2} , $$
defined for $u \in \mathcal{C}^{\infty}(M,\mathbb{R})$ positive.
Then the functional $\mathrm{EH}^t_p$ is bounded from below, namely
$$
\mathrm{EH}^t_p(u) \geq - \|\mu^t(\omega)\|_{L^{\frac{p}{p-2}}} \,\, ,
$$
for every $2 \leq p \leq \infty$ and for every $t > 0$. Moreover, when the exponent $p$ satisfies
$$
2 \leq p < q = \frac{2n}{n-1} \,\, ,
$$
the functional $\mathrm{EH}^t_p$ admits a minimum.
\end{proposition}

\section{Deformed Yamabe invariants} \label{sec:invariant}
\setcounter{equation} 0

\subsection{The deformed Yamabe invariants} \hfill \par

Thanks to the first part of Proposition \ref{prop:minimizer}, we can define the \emph{deformed Yamabe invariants} as follows.

\begin{definition}\label{def:lambda-t}
Let $(M^{2n},J)$ be a compact complex manifold of complex dimension $n$. For every Hermitian metric $\omega$ on $(M^{2n},J)$ and for every $t > 0$, we define the \emph{$t$-deformed Yamabe invariant of $\{\omega\}$} by
\begin{equation}\label{eq:Yam_inv}
\lambda_{(M,J)}^{t}(\{\omega\}) \coloneqq \inf \left\{\mathrm{EH}^{t}(\omega') : \omega' \in \{\omega\}\right\} \,\, ,
\end{equation}
where $\mathrm{EH}^{t}$ is the $t$-deformed Einstein--Hilbert functional as in \eqref{eq:EH-t}.
\end{definition}

The deformed Yamabe invariants are related to the Gauduchon degree and to the classical Yamabe invariant. First, arguing as in the proof of
\cite[Theorem 1.2]{LLS25}, one obtains the following properties.

\begin{proposition} \label{prop:Gaud_Yamabe}
Let $(M^{2n},J,\omega)$ be a compact Hermitian manifold. Then
\begin{equation} \label{eq:lambda_gaud_ineq}
\lambda^{t}_{(M,J)}(\{\omega\}) \leq \Gamma_{(M,J)}(\{\omega\}) \quad \text{for every $0 < t \leq 2n$} \,\, .
\end{equation}
Moreover, if $\eta \in \{\omega\}$ is the Gauduchon metric of volume $1$ in the conformal class of $\omega$ and $t=2n$, the following hold.
\begin{itemize}
\item[i)] If $\lambda^{2n}_{(M,J)}(\{\omega\}) = \Gamma_{(M,J)}(\{\omega\})$, then $\mathrm{scal}^{\mathrm{Ch}}(\eta)$ is constant.
\item[ii)] When $\lambda^{2n}_{(M,J)}(\{\omega\}) \leq 0$, for instance when $\Gamma_{(M,J)}(\{\omega\}) \leq 0$, the converse also holds, namely: $\lambda^{2n}_{(M,J)}(\{\omega\}) = \Gamma_{(M,J)}(\{\omega\})$ if and only if $\mathrm{scal}^{\mathrm{Ch}}(\eta)$ is constant.
\end{itemize}
\end{proposition}

\begin{proof}
Inequality \eqref{eq:lambda_gaud_ineq} follows directly from \eqref{eq:mu_t}, \eqref{eq:EHt(u)}, \eqref{eq:Gdeg} and \eqref{eq:Yam_inv}. Indeed, choosing $\eta$ as reference metric and noticing that $\mu^t(\eta) \leq \mathrm{scal}^{\mathrm{Ch}}(\eta)$ for $0 < t\leq 2n$, we obtain
$$
\lambda^{t}_{(M,J)}(\{\omega\}) \leq
\mathrm{EH}^{t}(1) =
\frac{\int_M \mu^{t}(\eta)
\frac{\eta^n}{n!}}{(\int_M \frac{\eta^n}{n!})^{\frac{2}{q}}} \leq
\int_M \mathrm{scal}^{\mathrm{Ch}}(\eta)\,\frac{\eta^n}{n!} =
\Gamma_{(M,J)}(\{\omega\}) \,\, .
$$
We now set $t=2n$ and notice that $\mu^{2n}(\eta) = \mathrm{scal}^{\mathrm{Ch}}(\eta)$. By the definition of $\lambda^{2n}_{(M,J)}(\{\omega\})$ in \eqref{eq:Yam_inv}, if 
$\lambda^{2n}_{(M,J)}(\{\omega\}) = \Gamma_{(M,J)}(\{\omega\})$, then $\eta$ realizes the infimum of $\mathrm{EH}^{2n}$ in the conformal class $\{\omega\}$. In particular, $\eta$ is a critical point of $\mathrm{EH}^{2n}$ and, by Theorem \ref{thm:variational}, $u=1$ solves \eqref{eq:t-DY} for $t=2n$ with respect to the form $\eta$. This is equivalent to saying that $\mathrm{scal}^{\mathrm{Ch}}(\eta)$ is constant.

The second claim follows from the uniqueness statement in the non-positive case, established in Lemma \ref{prop:rigidity}.
\end{proof}

\begin{remark}
Note that, in general, the Gauduchon degree $\Gamma_{(M,J)}(\{\omega\})$ does not dominate the $t$-deformed Yamabe invariant $\lambda_{(M,J)}^{t}(\{\omega\})$ for $t>2n$, due to the presence of terms involving $|T|^2$.
\end{remark}

Secondly, in the \emph{globally conformally K{\"a}hler} case, one can obtain explicit relations between the deformed Yamabe invariants and the classical Riemannian Yamabe invariant.

\begin{proposition} \label{prop:bound-classical-yamabe}
Let $(M^{2n},J,g)$ be a compact K{\"a}hler manifold. Let $\omega$ be the corresponding fundamental $2$-form and denote by $\lambda^{\mathrm{Yam}}_M(\{g\})$ the classical Yamabe invariant of the conformal class $\{g\}$. Then:
\begin{align}
&\lambda^{\mathrm{Yam}}_M(\{g\}) \leq \lambda^{t}_{(M,J)}(\{\omega\})  \quad \text{for every $t \geq 2n-1$} , \label{eq:lambda_Yam1} \\
&\lambda^{\mathrm{Yam}}_M(\{g\}) \geq \lambda^{t}_{(M,J)}(\{\omega\})  \quad \text{for every $0 < t \leq 2n-1$} . \label{eq:lambda_Yam2}
\end{align}
Moreover, if $g$ is a Yamabe minimiser in its conformal class, for instance if $g$ is K{\"a}hler--Einstein, then
\begin{equation} \label{eq:lambda_Yam3}
\lambda^{\mathrm{Yam}}_M(\{g\}) = \lambda^{t}_{(M,J)}(\{\omega\}) \quad \text{for every $2n-1 \leq t \leq 2n$} \,\, .
\end{equation}
\end{proposition}

\begin{proof}
If $g$ is K{\"a}hler and $t \geq 2n-1$, then
$$
\mu^{t}(\omega) = \mathrm{scal}(g) \,\, , \quad 
\frac{2t}{n-1} \geq \frac{2(2n-1)}{n-1} \,\, .
$$
Hence, for every positive smooth function $u: M \to \mathbb{R}$, we have
$$
\mathrm{EH}^{t}(u) \geq \frac{\int_M \left(u^2 \mathrm{scal}(g) + \frac{2(2n-1)}{n-1} |\nabla u|^2 \right)\, \frac{\omega^n}{n!}}{\left(\int_M u^q \, \frac{\omega^n}{n!}\right)^{\frac{2}{q}}} \,\, .
$$
Since the right-hand side is precisely the classical Einstein--Hilbert functional, taking the infimum over $u$ gives \eqref{eq:lambda_Yam1}. The same argument shows that \eqref{eq:lambda_Yam2} holds.

As for the second statement, assume that $g$ is K{\"a}hler, has unit volume, and is a Yamabe minimiser in its conformal class. Then
$$
\lambda^{\mathrm{Yam}}_M(\{g\}) =
\int_M \mathrm{scal}(g) \frac{\omega^n}{n!} =
\Gamma_{(M,J)}(\{\omega\}) \,\, ,
$$
and so \eqref{eq:lambda_Yam3} follows from \eqref{eq:lambda_gaud_ineq} and \eqref{eq:lambda_Yam1}.
\end{proof}

\begin{remark}
As a direct consequence of Proposition \ref{prop:bound-classical-yamabe}, if $(M^{2n},J,g_{\mathrm{KE}})$ is a compact K{\"a}hler--Einstein manifold with fundamental $2$-form $\omega_{\mathrm{KE}}$, then
$$
\lambda^{t}_{(M,J)}(\{\omega_{\mathrm{KE}}\}) = \mathrm{scal}(g_{\mathrm{KE}})\,\mathrm{Vol}(M,g_{\mathrm{KE}})^{\frac{1}{n}} \quad \text{for every $2n-1 \leq t \leq 2n$} \,\, .
$$
\end{remark}

\subsection{Existence of constant deformed scalar curvature metrics} \label{sec:aubin} \hfill \par

For notational convenience, we set
\begin{equation} \label{eq:Lambda_n}
\Lambda_n \coloneqq 2n\,\mathrm{Vol}(S^{2n})^{\frac{1}{n}} \,\, .
\end{equation}
Then \cite[Theorem 2.1]{Hebey_ICTP}, after \cite[Th{\'e}or{\`e}me 1]{MR0431287}, directly yields a \emph{subcritical existence result} for minimising solutions to the critical equation \eqref{eq:t-DY}.

\begin{theorem}[{see \cite{MR0431287, Hebey_ICTP}}] \label{thm:existence_subcritical}
Let $(M^{2n},J,\omega)$ be a compact Hermitian manifold of complex dimension $n$. Let $t>0$ and let $\Lambda_n$ as in \eqref{eq:Lambda_n}. Then
\begin{equation} \label{eq:general_ineq_yamabe}
\lambda^{t}_{(M,J)}(\{\omega\}) \leq t\,\Lambda_n \,\, .
\end{equation}
Moreover, if $\lambda^{t}_{(M,J)}(\{\omega\}) < t\,\Lambda_n$, then there exists a positive smooth solution to \eqref{eq:t-DY}, which is a minimiser of the $t$-deformed Einstein--Hilbert functional $\mathrm{EH}^t$ restricted to the conformal class $\{\omega\}$ as in \eqref{eq:EHt(u)}.
\end{theorem}

We notice that Proposition \ref{prop:Gaud_Yamabe} and Theorem \ref{thm:existence_subcritical} immediately imply the following existence result.

\begin{corollary}
Let $(M^{2n},J,\omega)$ be a compact Hermitian manifold of complex dimension $n$, and let $0 < t \leq 2n$. If $\Gamma_{(M,J)}(\{\omega\}) < t\,\Lambda_{n}$, for instance if $\Gamma_{(M,J)}(\{\omega\}) \leq 0$, then there exists a constant $t$-deformed scalar curvature metric in the conformal class $\{\omega\}$.
\end{corollary}

Moreover, Proposition \ref{prop:rigidity} and Theorem \ref{thm:existence_subcritical} allow us to calculate the $t$-deformed Yamabe invariant in the non-positive case, as the next corollary shows.

\begin{corollary}\label{cor:Gamma}
Let $(M^{2n},J,\omega)$ be a compact Hermitian manifold of complex dimension $n$. Let $t>0$, and let $c \leq 0$ be a non-positive constant. If $\mu^{t}(\omega) = c$, then 
$$
\lambda^{t}_{(M,J)}(\{\omega\}) = c\,\mathrm{Vol}(M,g_{\omega})^{\frac{1}{n}} \,\, .
$$
\end{corollary}

\begin{proof}
Since $\mu^{t}(\omega) = c \leq 0$, we have
$$
\lambda^{t}_{(M,J)}(\{\omega\}) \leq \mathrm{EH}^{t}(\omega) = c\,\mathrm{Vol}(M,g_{\omega})^{\frac{1}{n}}
\leq 0 \,\, ,
$$
and therefore Theorem \ref{thm:existence_subcritical} implies that the infimum of $\mathrm{EH}^{t}$ in the conformal class $\{\omega\}$ is realised by a smooth positive minimiser. By Proposition \ref{prop:rigidity}, such a minimiser is necessarily a constant rescaling of $\omega$, and this concludes the proof.
\end{proof}

It is still not clear whether equality in \eqref{eq:general_ineq_yamabe} can ever be realised. However, the following result gives a pointwise obstruction to this equality in terms of the sign of the $0$-deformed scalar curvature
$$ \mu^0(\omega) = \mathrm{scal}^{\mathrm{Ch}}(\omega) -\frac{n}{n-1}d^*\theta - \frac{n}{2}|T|^2 $$
at a point.

\begin{theorem}[\cite{MR0431287}] \label{thm:aubins_crit}
Let $(M^{2n},J,\omega)$ be a compact Hermitian manifold of complex dimension $n\geq 3$, and fix $t>0$.
If there exists a point $p \in M$ such that
$$
\left.(t-2n+1)\,\mu^0(\omega)\right|_p >0 \,\, ,
$$
then $\lambda^{t}_{(M,J)}(\{\omega\}) < t\,\Lambda_n$.
\end{theorem}

Notice that Theorem \ref{thm:aubins_crit} follows from a classical argument as in \cite[Th{\'e}or{\`e}me 4]{MR0431287} (see also \cite[Proposition~$1$]{MR2272821}). The argument consists in evaluating the $t$-deformed Einstein--Hilbert functional \eqref{eq:EHt(u)} on a suitable bump function, concentrated in a ball of very small radius around a point $p \in M$, and then expanding in Taylor series at $p$ as the radius goes to $0$. For the sake of completeness, we give a proof of Theorem \ref{thm:aubins_crit} in the Appendix, making use of \emph{holomorphic normal coordinates}.

\begin{remark}
As shown in \eqref{eq:mu_t_conformal}, the quantity $\mu^0(\omega)$ is conformally equivariant, namely
$$
\mu^0(e^f\omega) = e^{-f}\mu^0(\omega) \quad \text{for all $f \in \mathcal{C}^{\infty}(M,\mathbb{R})$} \,\, .
$$
Therefore, the sign of $\mu^0(\omega)$ at a point $p \in M$ is conformally invariant. Moreover, when $t = 2n-1$, which corresponds to the classical Yamabe problem, Theorem \ref{thm:aubins_crit} gives no information.
\end{remark}

\begin{remark} \label{rem:Hopf}
The condition stated in Theorem \ref{thm:aubins_crit} is far from being necessary. For example, consider the Hopf manifold $(X^n,\omega) = (S^1\times S^{2n-1},J,g)$ endowed with its standard Hermitian structure. Namely, $g$ is the product of the round metric of radius $1$ on $S^1$ and the round metric of radius $1$ on $S^{2n-1}$, while $J$ coincides with the transverse complex structure on $S^{2n-1}$ and sends the Reeb vector field of $S^{2n-1}$ to the generator of $S^1$. Then, $(X^n,\omega)$ admits a transitive left action of $\mathsf{U}(1) \times \mathsf{U}(n)$ by holomorphic isometries, and so $\mu^t(\omega)$ is constant for every $t \in \mathbb{R}$. Moreover, a direct computation (see Example \ref{ex:Hopf}) shows that
$$
\mu^t(\omega) = t(n-1)
$$
and so, in particular, $\mu^0(\omega) = 0$. On the other hand, using $\omega$ as a reference metric, for $n \geq 3$ we have
$$
\mathrm{EH}^t(1) = t\,\Lambda_n\frac{n-1}{2n}\frac{\mathrm{Vol}(S^1 \times S^{2n-1})^{\frac{1}{n}}}{\mathrm{Vol}(S^{2n})^{\frac{1}{n}}} \leq t\,\Lambda_n \frac{n-1}{2n} \left(\frac{15\pi}{8}\right)^{\frac{1}{3}}
$$
and so
$$
\lambda^{t}_{X^n}(\{\omega\}) < \frac{\sqrt[3]{15\pi}}{4}\,t\,\Lambda_n \,\, .
$$
\end{remark}

In the locally conformally K{\"a}hler setting, Theorem \ref{thm:aubins_crit} gives the following.

\begin{corollary} \label{cor:existence}
Let $(M^{2n},J,\eta)$ be a locally conformally K{\"a}hler manifold. Assume that $\eta$ is Gauduchon with volume $1$ and denote by $\theta$ its Lee form. If either
\begin{equation} \label{eq:lck-cond1}
t > 2n-1 \quad \text{and} \quad \Gamma_{(M,J)}(\{\eta\}) > \frac{n}{n-1}\|\theta\|_{L^2}^2 \,\, ,
\end{equation}
or
\begin{equation} \label{eq:lck-cond2}
\Gamma_{(M,J)}(\{\eta\}) < t\,\Lambda_n -\frac{t-2n}{2(n-1)}\|\theta\|_{L^2}^2 \,\, ,
\end{equation}
then there exists a $t$-deformed constant scalar curvature metric $\omega \in \{\eta\}$. In particular, if $t > 2n-1$ and $\eta$ is K{\"a}hler, then there exists a $t$-deformed constant scalar curvature metric $\omega \in \{\eta\}$.
\end{corollary}

\begin{proof}
Since $\eta$ is locally conformally K{\"a}hler and Gauduchon, it follows that
$$
|T|^2 = \frac{2}{n-1} |\theta|^2 \,\, , \quad d^*\theta = 0  \,\, .
$$
Hence, by \eqref{eq:mu_t}, we obtain
$$
\mu^t(\eta) = \mathrm{scal}^{\mathrm{Ch}}(\eta) +\frac{t -2n}{2(n-1)}|\theta|^2 \,\, .
$$
Therefore, if \eqref{eq:lck-cond1} holds, then there exists a point $p \in M$ such that
$$
\left.(t-2n+1)\,\mu^0(\omega)\right|_p >0 \,\, ,
$$
and the claim follows from Theorem \ref{thm:existence_subcritical} and Theorem \ref{thm:aubins_crit}. On the other hand, using $\eta$ as a reference metric, we have
$$
\mathrm{EH}^t(1) = \int_M \mu^t(\eta)\frac{\eta^n}{n!} = \Gamma_{(M,J)}(\{\eta\}) +\frac{t -2n}{2(n-1)}\|\theta\|_{L^2}^2.
$$
Thus, if \eqref{eq:lck-cond2} holds, then $\lambda^{t}_{(M,J)}(\{\eta\}) < t\,\Lambda_n$, and the claim follows again from Theorem \ref{thm:existence_subcritical}.

Finally, we observe that, when $t>2n-1$, the two ranges in \eqref{eq:lck-cond1} and \eqref{eq:lck-cond2} overlap if and only if
$$
\|\theta\|_{L^2}^2 < 2(n-1) \Lambda_n \,\, .
$$
This proves the last claim.
\end{proof}

\begin{remark}
As in the proof of Corollary \ref{cor:existence}, the two ranges in \eqref{eq:lck-cond1} and \eqref{eq:lck-cond2} may overlap beyond the K{\"a}hler case, depending on the size of $\|\theta\|_{L^2}$.
\end{remark}

\section{Examples} \label{sec:examples}
\setcounter{equation} 0

\subsection{Product of complex manifolds with a K{\"a}hler manifold} \hfill \par

Since the condition in Theorem \ref{thm:aubins_crit} is pointwise, one can construct examples of manifolds admitting metrics with constant $t$-deformed scalar curvature, for $t \neq 2n-1$, by taking products with K{\"a}hler manifolds. More precisely, the following result holds.

\begin{theorem}\label{thm:prod-kahler}
Let $(X^m,\omega_X)$ be a compact K{\"a}hler manifold of complex dimension $m$, let $(Y^k,\omega_Y)$ be a compact Hermitian manifold of complex dimension $k$, with $n \coloneqq m+k \geq 3$, and let $t>0$.
\begin{itemize}
\item[$i)$] If $t > 2n-1$ and there exists $x \in X$ such that $\mathrm{scal}(\omega_X)(x) > 0$, then $X \times Y$ admits a Hermitian metric with constant $t$-deformed scalar curvature.
\item[$ii)$] If $0 < t < 2n-1$ and there exists $x \in X$ such that $\mathrm{scal}(\omega_X)(x) < 0$, then $X \times Y$ admits a Hermitian metric with constant $t$-deformed scalar curvature.
\end{itemize}
\end{theorem}

\begin{proof}
Let us denote by $T_Y$, $\theta_Y$, and $d^*_Y$ the torsion, the Lee form, and the codifferential of $(Y^k,\omega_Y)$, respectively. Consider the Hermitian product
$$
(M^{2n},J) \coloneqq X^m \times Y^k \,\, , 
\qquad 
\omega_{\varepsilon} \coloneqq \varepsilon\, \omega_X \oplus \omega_Y 
\quad \text{for $\varepsilon >0$} \,\, ,
$$
and denote by $T_{\varepsilon}$, $\theta_{\varepsilon}$, and $d^*_{\varepsilon}$ the torsion, the Lee form, and the codifferential of $(M^{2n},J,\omega_{\varepsilon})$, respectively. Here and in the following, with a slight abuse of notation, we omit the pullback symbols from the projections of $M$ onto its factors.

By \eqref{eq:T}, the torsion $T_{\varepsilon}$ splits diagonally, that is,
$$
T_{\varepsilon} = 0 \oplus T_Y.
$$
Thus, \eqref{eq:theta} implies that
$$
\theta_{\varepsilon} = 0 \oplus \theta_Y
$$
and hence
$$
d^*_{\varepsilon}\theta_{\varepsilon} = d^*_Y\theta_Y \,\, .
$$
Moreover, the Chern scalar curvature of $\omega_{\varepsilon}$ is given by
$$
\mathrm{scal}^{\mathrm{Ch}}(\omega_{\varepsilon}) 
= \varepsilon^{-1}\,\mathrm{scal}(\omega_X) 
+ \mathrm{scal}^{\mathrm{Ch}}(\omega_Y) \,\, .
$$
Therefore, by \eqref{eq:mu_t} we obtain
\begin{equation} \label{eq:mu^0(omegaepsilon)}
\mu^0(\omega_{\varepsilon}) =
\varepsilon^{-1}\,\mathrm{scal}(\omega_X) + \mathrm{scal}^{\mathrm{Ch}}(\omega_Y)
-\frac{n}{n-1}d^*_Y\theta_Y
-\frac{n}{2}|T_Y|^2 \,\, .
\end{equation}
The result then follows from Theorem \ref{thm:aubins_crit} and \eqref{eq:mu^0(omegaepsilon)}, after choosing $\varepsilon>0$ sufficiently small so that $\mu^0(\omega_\varepsilon)$ has the required sign at some point.
\end{proof}

\subsection{Locally homogeneous Hermitian manifolds} \hfill \par

We recall that a Hermitian manifold $(M,J,g)$ is said to be \emph{locally homogeneous} if the action of its pseudogroup of local holomorphic isometries is transitive. Namely, for every $x,y \in M$, there exist neighborhoods $U_x$ and $U_y$ of $x$ and $y$, respectively, and a local holomorphic isometry $f \colon U_x \to U_y$ such that $f(x)=y$.

Clearly, every locally homogeneous Hermitian manifold has constant deformed scalar curvatures.

For example, consider an Inoue--Bombieri surface of type $S^0$, which is a compact complex surface with negative Kodaira dimension and first Betti number equal to $1$. The canonical metric introduced by Tricerri \cite{MR0706055} is invariant under left-translations on the covering solvable Lie group, hence it is locally homogeneous. This yields an example of a non-Vaisman, locally conformally K{\"a}hler structure with constant deformed $t$-curvature for all $t \in \mathbb{R}$.

Another example is the Hopf manifold, obtained as the quotient of $\mathbb{C}^n \setminus \{0\}$ by a free and properly discontinuous action of $\mathbb{Z}$ (see Example \ref{ex:Hopf}).

\subsection{Hirzebruch surfaces} \hfill \par

The first non-homogeneous example of a Riemannian Einstein manifold with positive scalar curvature was constructed by Page \cite{Page78, MR879886} on $\mathbb{CP}^2 \# \overline{\mathbb{CP}}{}^2$, and later generalized by B\'erard-Bergery \cite{MR727843}, using a cohomogeneity-one method. The same approach was used in \cite{MR4567558} to construct non-homogeneous Hermitian manifolds with constant Chern scalar curvature. Here, we show that this framework can be used to construct non-homogeneous Hermitian manifolds with constant $t$-deformed scalar curvature. In what follows, we specialize to Hirzebruch surfaces and to the case $t=2n=4$, that is, to the momentum $\mu^4$ defined in \eqref{eq:moment}.

To this end, let $m \in \mathbb{Z}$, $m>0$. We recall that the $m$-Hirzebruch surface $M_m$ is the total space of the $\mathbb{CP}^1$-bundle $\pi_m \colon M_m \to \mathbb{CP}^1$ associated with the $\mathsf{U}(1)$-principal bundle $s_m \colon \Sigma_m = S^3/\mathbb{Z}_m \to \mathbb{CP}^1$, whose Euler class is
$$
e(\Sigma_m) = -\frac{m}{2\pi}[\omega_{\mathrm{FS}}] \,\, ,
$$
via the standard left action of $\mathsf{U}(1)$ on $\mathbb{CP}^1$. Here, we denoted by $\omega_{\mathrm{FS}}$ the fundamental $2$-form associated to the Fubini-Study metric $g_{\mathrm{FS}}$ on $\mathbb{CP}^1$, normalized so that $\mathrm{Ric}(g_{\mathrm{FS}}) = 2\,g_{\mathrm{FS}}$. The complex surface $M_m$ is compact and simply-connected. Moreover, it admits a left holomorphic action of $\mathsf{U}(2)$, with cohomogeneity one principal orbits, that turns $\pi_m$ into a $\mathsf{U}(2)$-equivariant holomorphic submersion. The subset $M_m^{\mathrm{reg}} \subset M_m$ of regular points for the action of $\mathsf{U}(2)$ is diffeomorphic to $\Sigma_m \times (0,L)$, with $L>0$. Finally, the action has precisely two singular orbits, both diffeomorphic to $\mathbb{CP}^1$.

Fix a principal connection $\theta_m$ on $\Sigma_m$ such that $d\theta_m = -m\,s_m^*\omega_{\mathrm{FS}}$. Given two positive numbers $\nu, \kappa >0$, we denote by $\tilde{g}(\nu,\kappa)$ the $\mathsf{U}(2)$-invariant Riemannian metric on $\Sigma_m$ characterized by the following properties: \begin{itemize}
\item[$i)$] the restriction of $\tilde{g}(\nu,\kappa)$ to the vertical distribution $\ker{( d s_m)} \subset T\Sigma_m$ corresponds to the round metric of length $\frac{2\pi}{m}\nu$ on $\mathsf{U}(1)$;
\item[$ii)$] the $\tilde{g}(\nu,\kappa)$-orthogonal complement $\ker{( d s_m)}^{\perp}\subset T\Sigma_m$ to the vertical distribution $\ker{( d s_m)}$ satisfies $\ker{( d s_m)}^{\perp} = \ker{(\theta_m)}$;
\item[$iii)$] the projection $s_m: (\Sigma_m,\tilde{g}(\nu,\kappa)) \to (B,\kappa^2g_{\mathrm{FS}})$ is a Riemannian submersion with totally geodesic fibers.
\end{itemize}
We refer to \cite{MR1030671} for further details on the geometry of torus bundles over the product of K{\"a}hler--Einstein manifolds with positive first Chern class.

Given two smooth, positive functions $f,h : (0,L) \to \mathbb{R}$, the formula
\begin{equation} \label{eq:cohom1}
g(f,h) \coloneqq  dr^2 +\tilde{g}(f(r),h(r))
\end{equation}
defines a $\mathsf{U}(2)$-invariant Hermitian metric on $M_m^{\mathrm{reg}}$, which extends smoothly to the whole manifold $M_m$ if and only if the following \emph{smoothness conditions} are satisfied: \begin{itemize}
\item[i)] $f$ is the restriction of a smooth odd function on $\mathbb{R}$ satisfying
$$
f(L+r)=-f(L-r) \quad \text{and} \quad f'(0)=1=-f'(L)  ,
$$
\item[ii)] $h$ is the restriction of a smooth even function on $\mathbb{R}$ satisfying $h(L+r) = h(L-r)$.
\end{itemize}

Hermitian metrics of the form \eqref{eq:cohom1} were studied in \cite{MR4567558}. Note that a different normalization of the K{\"a}hler--Einstein metric on $\mathbb{CP}^1$ is used in \cite{MR4567558}, which leads to slightly different constants in the corresponding formulas.

By \cite[Theorem A and (3.11)]{MR4567558}, the metric $g(f,h)$ is globally conformally K{\"a}hler, and it is K{\"a}hler if and only if
\begin{equation}
2h(r)h'(r) +mf(r) = 0  .
\end{equation}
Moreover, by \cite[Proposition 3.15]{MR4567558}, the Chern scalar curvature of $g(f,h)$ is given by
\begin{equation} \label{eq:scal}
\begin{split}
\mathrm{scal}^{\mathrm{Ch}}(r) =& -2\tfrac{f''(r)}{f(r)} -2\tfrac{h''(r)}{h(r)} +2\left(\tfrac{h'(r)}{h(r)} -\tfrac{f'(r)}{f(r)}\right)\tfrac{h'(r)}{h(r)} +4\tfrac1{h(r)^2} \\
& +2m\left(f'(r) +f(r)\tfrac{h'(r)}{h(r)}\right) \tfrac1{h(r)^2}   ,
\end{split}
\end{equation}
while, by the proof of \cite[Proposition 4.6]{MR4567558}, the codifferential of the Lee form is given by
\begin{equation} \label{eq:codiffLee}
d^*\theta(r) = \tfrac1{h(r)^2} \Big(\big(2h(r)h'(r) +mf(r)\big)\tfrac{f'(r)}{f(r)} +\big(2h(r)h''(r) +2h'(r)^2 +mf'(r)\big)\Big)  .
\end{equation}

The main result of this section is the following existence theorem.

\begin{theorem}\label{thm:ex}
For every integer $m >0$, the $m$-Hirzebruch surface $M_m$ admits a $\mathsf{U}(2)$-invariant globally conformally K{\"a}hler metric with positive constant momentum.
\end{theorem}

\begin{proof}
Consider the ansatz
\begin{equation} \label{eq:ans1}
f_{\phi}(r) \coloneqq -\frac{1}{m}\phi'(r)  , \quad h_{\phi}(r) \coloneqq \sqrt{\frac52\phi(r)}
\end{equation}
for some smooth, decreasing, positive, function $\phi: (0,L) \to \mathbb{R}$. Notice that
$$
2h_{\phi}(r)h_{\phi}'(r) +mf_{\phi}(r) = \frac32\phi'(r) < 0 \quad \text{for all $0 < r < L$}
$$
and so the corresponding metric $g(f_{\phi},h_{\phi})$ is necessarily non-K{\"a}hler. Then, by \eqref{eq:moment}, \eqref{eq:scal} and \eqref{eq:codiffLee}, the equation
$$
\mu^4\big(g(f_{\phi},h_{\phi})\big) = c  , \quad c \in \mathbb{R}
$$
becomes
\begin{equation} \label{eq:c1costmu}
10\,\phi(r)^2\frac{\phi'''(r)}{\phi'(r)} +2\,\phi(r)\phi''(r) -3\,\phi'(r)^2 +5c\,\phi(r)^2 -8\,\phi(r) = 0  .
\end{equation}
Let $0 < k < 1$, to be fixed later, and set $\phi(0) = 1$, $\phi(L) = k$. Then, the  smoothness conditions for $f,h$ imply that
\begin{equation} \label{eq:boundcond} \begin{gathered}
\phi(0)=1  , \quad \phi'(0)=0  , \quad \phi''(0)=-q  , \\
\phi(L)=k  , \quad \phi'(L)=0  , \quad \phi''(L)=q  .
\end{gathered}\end{equation}
We remark that, if there exists a smooth solution $\phi: [0,L] \to \mathbb{R}$ to \eqref{eq:c1costmu} satisfying the boundary conditions \eqref{eq:boundcond}, then it can be extended to a smooth even function on $\mathbb{R}$ satisfying $\phi(L + r)=\phi(L - r)$. Assume further that $\phi$ is a solution of the ODE
\begin{equation} \label{eq:ans2}
\phi'(r)=-\sqrt{u(\phi(r))}
\end{equation}
for some smooth, positive real function $u=u(t)$. Then, under the change of variable $t=\phi(r)$, the equation \eqref{eq:c1costmu} becomes
\begin{equation} \label{eq:c1costmu'}
5t^2u''(t) +tu'(t) -3u(t) +5ct^2 -8t =0  ,
\end{equation}
which can be explicitly integrated:
\begin{equation} \label{eq:sol-u}
u_{a,b,c}(t) = a\, t^{\frac{2-\sqrt{19}}{5}} -4t +b\, t^{\frac{2+\sqrt{19}}{5}}  -\frac59ct^2  , \quad a, b, c \in \mathbb{R}  .
\end{equation}
Notice that, by \eqref{eq:ans2}, equation \eqref{eq:boundcond} implies that the corresponding solution to \eqref{eq:c1costmu'} has to verify
$$
u(1)=0  , \quad u(k)=0  , \quad u'(1)=-2m   , \quad u'(k)=2m  .
$$
By imposing the first three conditions
$$
u_{a,b,c}(1)=0  , \quad u_{a,b,c}(k)=0  , \quad u_{a,b,c}'(1)=-2m
$$
at the general solution \eqref{eq:sol-u}, we obtain the following three values $a(k), b(k), c(k)$ depending on $k$:
$$\begin{aligned}
a(k) &\coloneqq \frac{(10m +4(\sqrt{19}-3))\,k^{\frac{8+\sqrt{19}}{5}} -10(m+2)\,k^{\frac{2\sqrt{19}}{5}}
+4(8 -\sqrt{19}) k^{\frac{3+\sqrt{19}}{5}}}{2\sqrt{19} k^{\frac{8+\sqrt{19}}{5}} -(8+\sqrt{19})\,k^{\frac{2\sqrt{19}}{5}} +(8-\sqrt{19})}  , \\
b(k) &\coloneqq \frac{-(8+\sqrt{19})\,a(k) +10(m+2)}{8-\sqrt{19}}  , \\
c(k) &\coloneqq \frac{9}{50}\big(-(\sqrt{19}-2)a(k) +(\sqrt{19}+2)b(k) +10(m-2)\big)  .
\end{aligned}$$
Set $u_k \coloneqq u_{a(k),b(k),c(k)}$ and observe that
$$
u_k'(k)-2m = \frac{\alpha(k)}{\beta(k)}  ,
$$
where
$$\begin{aligned}
\alpha(k) \coloneqq &\,
2\big((8-\sqrt{19})m -2(\sqrt{19}+1)\big)\,k^{1+\frac{2\sqrt{19}}{5}} -4\sqrt{19}(m-2)\,k^{\frac{8+\sqrt{19}}{5}} \\
& +2\big((8+\sqrt{19})m-2(\sqrt{19}-1)\big)\,k^{\frac{2\sqrt{19}}{5}} -2\big((8+\sqrt{19})m+2(\sqrt{19}-1)\big)\,k \\
& +4\sqrt{19}(m+2)\,k^{\frac{\sqrt{19}-3}{5}} -2\big((8-\sqrt{19})\,m+2(\sqrt{19}+1)\big)
\end{aligned}$$
and
$$
\beta(k) \coloneqq 2\sqrt{19}\,k^{\frac{8+\sqrt{19}}{5}} -(8+\sqrt{19})\,k^{\frac{2\sqrt{19}}{5}} +(8-\sqrt{19})  .
$$
We notice that
$$\begin{gathered}
\alpha(1) = \alpha'(1) = \alpha''(1) = 0  , \\
\alpha'''(1) = -\frac{36\sqrt{19}}{25}m < 0  , \\
\alpha(0) = -2\big(2(\sqrt{19}+1)+(8-\sqrt{19})\,m\big) < 0  ,
\end{gathered}$$
and so there exists $k_* \in (0,1)$ such that $\alpha(k_*) = 0$. Moreover
$$\begin{gathered}
\beta(0) = 8-\sqrt{19} > 0  , \quad \beta(1) = 0  , \\
\beta'(k) = -\frac{2\sqrt{19}(8+\sqrt{19})}{5}\big(1-k^{\frac{8-\sqrt{19}}{5}}\big)k^{\frac{2\sqrt{19}-5}{5}} < 0 \quad \text{for all $0 \leq k < 1$}
\end{gathered}$$
and so $\beta(k) >0$ for all $0 < k < 1$. In particular, $\beta(k_*) >0$.

Therefore, the function
$$
u_* : [k_*,1] \to \mathbb{R}  , \quad u_*(t) \coloneqq  u_{k^*}(t)
$$
verifies 
\begin{equation} \label{eq:ODEu*}
5t^2u_*''(t) +tu_*'(t) -3u_*(t) +5c_*t^2 -8t =0 \quad \text{for all $t \in [k_*,1]$}  ,
\end{equation}
with $c_* \coloneqq c(k_*)$, and
\begin{equation} \label{eq:boundODEu*}
u_*(k_*)=0  , \quad u_*(1)=0  , \quad u_*'(k_*)=2m   , \quad u_*'(1)=-2m  .
\end{equation}
By the boundary conditions \eqref{eq:boundODEu*}, it follows that there exists a maximum point $t_0 \in (k_*,1)$ for $u_*$ such that $u_*'(t) \geq 0$ for all $t \in [k_*,t_0]$ and $u_*(t_0)>0$. Therefore, from \eqref{eq:ODEu*} we get
$$
5c_*t_0^2 = -5t_0^2u_*''(t_0) +3u_*(t_0) +8t_0 > 8t_0
$$
and so $c_*>0$ and $t_0 > \frac{8}{5c_*}$.

We claim that $u_*(t)>0$ for all $t \in (k_*,1)$. Indeed, assume by contradiction that the function $u_*$ is not positive on $(k_*,1)$. Then, by the boundary conditions \eqref{eq:boundODEu*}, it follows that there exists a minimum point $t_1 \in (t_0,1)$ for $u_*$ such that $u_*(t_1) \leq 0$. Then, from \eqref{eq:ODEu*} we get
$$
0 \leq 5t_1^2u_*''(t_1) = 3u_*(t_1) -5c_*t_1^2 +8t_1 \leq t_1(8 -5c_*t_1)
$$
and so $t_1 \leq \frac{8}{5c_*} < t_0$, which contradicts the fact that $t_0 < t_1$.

Therefore, the function
$$
\varphi_*: [k_*,1] \to [0,L]  , \quad \varphi_*(r) \coloneqq \int_r^1 \frac{ d  t}{\sqrt{u_*(t)}}
$$
is smooth, decreasing and its inverse $\phi_* \coloneqq \varphi_*^{-1}$ solves the ODE \eqref{eq:c1costmu} with the boundary conditions \eqref{eq:boundcond}. Therefore, the $\mathsf{U}(2)$-invariant metric $g(f_{\phi_*},h_{\phi_*})$ is globally conformally K{\"a}hler, extends smoothly over the singular orbits, and satisfies $\mu^4\big(g(f_{\phi_*},h_{\phi_*})\big) = c_*>0$.
\end{proof}

\appendix

\section{}
\setcounter{equation} 0

\subsection{Local holomorphic coordinates} \hfill \par

Let $(M,J,\omega)$ be a compact Hermitian manifold. Fix a choice of local holomorphic coordinates $\{z^i\}_{i=1,{\dots},n}$ on $(M,J)$, and write
$$
\omega = h_{i\bar{j}}\,\sqrt{-1}\,dz^i \wedge d\bar{z}^{j} \,\, .
$$
Then the Chern curvature tensor $\Theta$ defined in \eqref{eq:Theta} is given by
$$
\Theta = \Theta_{i\bar{j}k\bar{\ell}}\,\sqrt{-1}\,dz^i \wedge d\bar{z}^{j} \otimes \sqrt{-1}\,dz^k \wedge d\bar{z}^{\ell} \,\, ,
$$
where
\begin{equation} \label{eq:Thetaijkl}
\Theta_{i\bar{j}k\bar{\ell}} = -\frac{\partial^2h_{k\bar{\ell}}}{\partial z^i \partial \bar{z}^j} +h^{\bar{q}p} \frac{\partial h_{p\bar{\ell}}}{\partial \bar{z}^j} \frac{\partial h_{k\bar{q}}}{\partial z^i} \,\, .
\end{equation}
The Chern scalar curvature $\mathrm{scal}^{\mathrm{Ch}}$ is then given by
\begin{equation} \label{eq:scalChcoord}
\mathrm{scal}^{\mathrm{Ch}} = 2\,h^{\bar{j}i}h^{\bar{\ell}k}\Theta_{i\bar{j}k\bar{\ell}} \,\, .
\end{equation}
We also mention that, because of the lack of the K{\"a}hler symmetries for $\Theta$, there is another natural contraction, locally defined as
\begin{equation} \label{eq:scalChcoord2}
\widetilde{\mathrm{scal}}{}^{\mathrm{Ch}} \coloneqq 2\,h^{\bar{\ell}i}h^{\bar{j}k}\Theta_{i\bar{j}k\bar{\ell}} \,\, .
\end{equation}
With the above normalization, \cite[Eq.\ (19)]{MR0742896} gives
\begin{equation} \label{eq:scalCh-scalCh}
\widetilde{\mathrm{scal}}{}^{\mathrm{Ch}} = \mathrm{scal}^{\mathrm{Ch}} -d^*\theta -|\theta|^2
\end{equation}
and hence, in particular, $\widetilde{\mathrm{scal}}{}^{\mathrm{Ch}} = \mathrm{scal}^{\mathrm{Ch}}$ if the metric $\omega$ is balanced.

The torsion $T$ defined in \eqref{eq:T} is given by
$$
T = T_{ij}^k \,dz^i \wedge dz^j \otimes \frac{\partial}{\partial z^k}
+\overline{T_{rs}^{\ell}} \,d\bar{z}^r \wedge d\bar{z}^s \otimes \frac{\partial}{\partial \bar{z}^{\ell}} \,\, ,
$$
where
\begin{equation} \label{eq:Tcoord}
T_{ij}^k = h^{\bar{\ell}k}\left(\frac{\partial h_{j\bar{\ell}}}{\partial z^i}-\frac{\partial h_{i\bar{\ell}}}{\partial z^j}\right) \,\, .
\end{equation}
Consequently, the Lee form $\theta$ defined in \eqref{eq:theta} is given by
$$
\theta = \theta_i \,dz^i +\overline{\theta_j} \,d\bar{z}^j \,\, , \quad \text{with} \quad \theta_i = T_{ik}^k \,\, ,
$$
and the norms of $T$ and $\theta$ are
\begin{equation} \label{eq:|T|coord}
|T|^2 = 2\,h^{\bar{j}i}h^{\bar{q}p}h_{r\bar{s}}T_{ip}^r\overline{T_{jq}^s} \,\, , \quad |\theta|^2 = 2\,h^{\bar{j}i} \theta_i\overline{\theta_j} \,\, .
\end{equation}

We record some explicit curvature computations on the classical Hopf manifold, which are used in Remark \ref{rem:Hopf}.

\begin{example}[The Hopf manifold] \label{ex:Hopf}
Fix any $\alpha \in \mathbb{C}$, with $0<|\alpha|<1$, and consider the action of $\mathbb{Z}$ on $\mathbb{C}^n$ generated by $z \mapsto \alpha z$. Then the \emph{classical Hopf manifold} is given by the quotient $X \coloneqq (\mathbb{C}^n \setminus \{0\})/\mathbb{Z}$. Both the canonical complex structure and the Hermitian metric
\begin{equation} \label{eq:Hopfmetric}
\omega \coloneqq |z|^{-2}\delta_i^j\,\sqrt{-1}\,dz^i\wedge d\bar{z}^j
\end{equation}
of $\mathbb{C}^n \setminus \{0\}$ are $\mathbb{Z}$-invariant, and so descend to the quotient and define a Vaisman structure on $X$ (see, \emph{e.g.}, \cite{MR0418003}). Notice that $X$ is diffeomorphic to $S^1 \times S^{2n-1}$ for every choice of $\alpha$, and that the standard Hermitian structure considered in Remark \ref{rem:Hopf} arises as a particular case of this construction. By using the standard coordinates $\{z_i\}_{i=1,{\dots},n}$ on $\mathbb{C}^n$, by \eqref{eq:Hopfmetric} one gets
$$
h_{i\bar{j}} = |z|^{-2}\delta_i^j \,\, , \quad
h^{\bar{j}i} = |z|^2\delta_j^i \,\, ,
$$
and so standard computations based on \eqref{eq:Thetaijkl} and \eqref{eq:Tcoord} give
$$
\Theta_{i\bar{j}k\bar{\ell}} = |z|^{-6}\Big(|z|^2\delta_{i}^{j} -\bar{z}^{i}z^{j}\Big)\delta_{k}^{\ell} \,\, , \quad T_{ij}^k = -|z|^{-2}\Big(\bar{z}^{i}\delta_{j}^{k} -\bar{z}^{j}\delta_{i}^{k}\Big) \,\, .
$$
Therefore, by \eqref{eq:scalChcoord}, \eqref{eq:|T|coord} and the fact that $\omega$ is Vaisman, one obtains
$$
\mathrm{scal}^{\mathrm{Ch}}(\omega) = 2n(n-1) \,\, , \quad
d^*\theta = 0 \,\, , \quad
|T|^2 = 4(n-1) \,\, .
$$
Hence, by \eqref{eq:mu_t}, the $t$-deformed scalar curvature of $\omega$ is
$$
\mu^t(\omega) = t\,(n-1) \,\, .
$$
\end{example}

We are now ready to prove Lemma \ref{lem:confchange}, namely, to derive the formulas for the conformal change of $\mathrm{scal}^{\mathrm{Ch}}$, $d^*\theta$, and $|T|^2$.

\begin{proof}[Proof of Lemma \ref{lem:confchange}]
Let $(M^{2n},J,\omega)$ be a Hermitian manifold, let $f:M \to \mathbb{R}$ be a smooth function and set $\omega' \coloneqq e^f \omega$. Fix a choice of local holomorphic coordinates $\{z^i\}_{i=1,{\dots},n}$ on $(M,J)$ and write
$$
\omega = h_{i\bar{j}}\,\sqrt{-1}\,dz^i \wedge d\bar{z}^{j} \,\, , \quad
\omega' = h'_{i\bar{j}}\,\sqrt{-1}\,dz^i \wedge d\bar{z}^{j} \,\, ,
$$
with
\begin{equation} \label{eq:hh'f}
h'_{i\bar{j}} = e^f\,h_{i\bar{j}} \quad \text{and} \quad
h'{}^{\bar{j}i} = e^{-f}\,h^{\bar{j}i} \,\, .
\end{equation}
A direct computation based on \eqref{eq:Thetaijkl}, \eqref{eq:scalChcoord} and \eqref{eq:hh'f} shows that
$$
\mathrm{scal}^{\mathrm{Ch}}(\omega') = e^{-f} \Big(\mathrm{scal}^{\mathrm{Ch}}(\omega) -2nh^{\bar{j}i}\frac{\partial^2 f}{\partial z^i\partial \bar{z}^j}\Big) \,\, .
$$
Moreover, by \cite[p.\ 502--503]{MR0742896}, we have 
$$
-2h^{\bar{j}i}\frac{\partial^2 f}{\partial z^i\partial \bar{z}^j} = \Delta f +g(\nabla f, \theta^{\sharp})
$$
and so \eqref{eq:confscalCh} follows. Analogously, \eqref{eq:confT} follows from a direct computation based on \eqref{eq:Tcoord}, \eqref{eq:|T|coord} and \eqref{eq:hh'f}.

Finally, to prove \eqref{eq:confcodiftheta}, we notice that the Lee form $\theta'$ of $\omega'$ is
$$
\theta' = \theta + (n-1) d f .
$$
and that
$$
\mathrm{div}_{g'}X = \mathrm{div}_{g}X + n\, d f(X)  , \quad \mathrm{div}_{g}(\phi X) = \phi\, \mathrm{div}_{g}(X) +  d \phi (X) \,\, .
$$
Therefore, we get
\begin{align*}
d^*_{g'}\theta' &= -\mathrm{div}_{g'}((\theta')^{\sharp_{g'}}) \\
&= -\mathrm{div}_{g}\left(e^{-f}\left(\theta^{\sharp_g}+ (n-1)\nabla^g f\right)\right) - ne^{-f} d f\left(\theta^{\sharp_g}+ (n-1)\nabla^g f\right) \\
&= e^{-f}\left(d^{*}_g\theta+(n-1)\Delta_gf\right) + e^{-f}(n-1)|\nabla^g f|_g^2 + e^{-f}g(\nabla^gf, \theta^{\sharp_{g}}) \\
&\qquad - ne^{-f} \left(g(\nabla^gf, \theta^{\sharp_{g}})+ (n-1)|\nabla^gf|_g^2\right) \\
&= e^{-f} \left(d^*_{g}\theta +(n-1)\Delta_gf -(n-1)g(\nabla^gf, \theta^{\sharp_{g}}) -(n-1)^2|\nabla^gf|_g^2 \right) \,\, ,
\end{align*}
which concludes the proof.
\end{proof}

\subsection{Proof of Theorem \ref{thm:aubins_crit}} \hfill \par

The aim of this section is to prove the following result, following the argument in \cite{MR0431287} and using a special choice of holomorphic coordinates adapted to the non-K{\"a}hler setting. For notational convenience, we denote by $\sigma_{2n}$ the square of the $2n$-dimensional best Sobolev constant, \emph{i.e.},
\begin{equation} \label{eq:Sobconst}
\frac{1}{\sigma_{2n}} \coloneqq n(n-1)\mathrm{Vol}(S^{2n}_1)^{\frac{1}{n}} \,\, .
\end{equation}

\begin{theorem}[\cite{MR0431287}] \label{thm:bubble_test}
Let $(M^{2n},J,\omega)$ be a compact Hermitian manifold of complex dimension $n$, with $n \geq 3$.
Fix a smooth function $\psi$ and a positive constant $k>0$, and consider the equation
\begin{equation}\label{eq:yamabegeneral}
k\Delta u + \psi\,u = \lambda\, u^{q-1} \,\, ,
\end{equation}
where $q = 2^* = \frac{2n}{n-1}$.
If there exists a point $p\in M$ such that
\begin{equation}
\frac{2(2n-1)}{n-1}\psi(p) - k\, \mathrm{scal}(g_\omega)_p < 0 \,\, ,
\end{equation}
then
\begin{equation} \label{eq:subcrit-appendix}
\inf_{\substack{u \in \mathcal{C}^{\infty}(M,\mathbb{R}) \\ u>0}}
\frac{\int_M\left(\psi u^2+k|\nabla u|^2\right)\frac{\omega^n}{n!}}
{\|u\|^2_{L^q}}
< \frac{k}{\sigma_{2n}} \,\, ,
\end{equation}
where $\sigma_{2n}$ is as in \eqref{eq:Sobconst}.
\end{theorem}

\begin{remark}
Notice that Theorem \ref{thm:aubins_crit} follows immediately from Theorem \ref{thm:bubble_test}. Indeed, if we set 
$$
\psi = \mu^t(\omega) \quad \text{and} \quad k = \frac{2t}{n-1} \,\, ,
$$
then, by using \eqref{eq:scal-scalCh}, \eqref{eq:mu_t} and \eqref{eq:general_ineq_yamabe}, a direct computation shows that
\begin{align*}
\frac{2(2n-1)}{n-1}\psi - k\, \mathrm{scal}(g_\omega)
= -\frac{2(t-2n+1)}{n-1}\mu^0(\omega) \,\, , \quad \frac{k}{\sigma_{2n}} = t\,\Lambda_n \,\, .
\end{align*}
This gives precisely the assertion of Theorem \ref{thm:aubins_crit}.
\end{remark}

In order to prove Theorem \ref{thm:bubble_test}, we recall the following result on the existence of special coordinates in the non-K{\"a}hler setting (see also \cite[Lemma 3.4]{MR3632564}).

\begin{lemma} \label{lem:locholnorm1}
Let $(M^{2n},J,\omega)$ be a Hermitian manifold. For every point $p \in M$, there exist local holomorphic coordinates for $(M,J)$ centred at $p$ such that the local component $h_{i\bar{j}}$ of $\omega$ verify
\begin{equation} \label{eq:normcoord}
h_{i\bar{j}}(0) = \delta_{i}^{j} \,\, , \quad
\frac{\partial h_{j\bar{\ell}}}{\partial z^i} (0) + \frac{\partial h_{i\bar{\ell}}}{\partial z^j} (0) = 0 \,\, .
\end{equation}
\end{lemma}

\begin{proof}
Let $\{z^i\}_{i=1,{\dots},n}$ be a choice of holomorphic coordinates centred at $p \in M$ such that $h_{i\bar{j}}(0)=\delta_{ij}$. Consider
$$
w^k \coloneqq z^k+\frac{1}{2}\left(\frac{\partial h_{s\bar{k}}}{\partial z^r} (0) + \frac{\partial h_{r\bar{k}}}{\partial z^s} (0)\right)z^rz^s
$$
and notice that, by the holomorphic inverse function theorem, after possibly shrinking the coordinate neighbourhood, $\{w^k\}$ is a system of local holomorphic coordinates centred at $p$. Denote by $\tilde{h}_{i\bar{j}}$ the metric components in these new coordinates. We have
$$
z^k = w^k -\frac{1}{2}\left(\frac{\partial h_{s\bar{k}}}{\partial z^r} (0) + \frac{\partial h_{r\bar{k}}}{\partial z^s} (0)\right)w^rw^s +O(|w|^3) \,\, ,
$$
and
$$
\tilde{h}_{i\bar{j}} = h_{k\bar{\ell}}\frac{\partial z^k}{\partial w^i}\frac{\partial \bar{z}^{\ell}}{\partial \bar{w}^{j}} \,\, .
$$
Therefore, we get
$$
\tilde{h}_{i\bar{j}}(0) = \delta_{k}^{\ell}\delta_{i}^{k}\delta_{\ell}^{j} = \delta_{i}^{j}
$$
and
$$
\frac{\partial \tilde{h}_{j\bar{\ell}}}{\partial w^i}(0) +\frac{\partial \tilde{h}_{i\bar{\ell}}}{\partial w^j}(0) =
\frac{\partial h_{j\bar{\ell}}}{\partial z^i}(0)
+\frac{\partial h_{i\bar{\ell}}}{\partial z^j}(0)
+2\frac{\partial^2 z^{\ell}}{\partial w^i\partial w^j}(0) = 0 \,\, ,
$$
which concludes the proof.
\end{proof}

Local coordinates satisfying the properties in Lemma \ref{lem:locholnorm1} are called \emph{holomorphic normal coordinates at $p$}. They naturally extend the classical holomorphic normal coordinates in the K{\"a}hler case, and in general they differ from Riemannian normal coordinates. 

The next lemma records the Taylor expansions of the Hermitian metric, its inverse, and its determinant in local holomorphic normal coordinates.

\begin{lemma} \label{lem:locholnorm2}
Let $(M^{2n},J,\omega)$ be a Hermitian manifold and let
$\{z^i\}_{i=1,\dots,n}$ be local holomorphic normal coordinates centred at a point $p \in M$. Then the following Taylor expansions at $p$ hold:
\begin{align}
h_{i\bar{j}}(z) &= \delta_{i}^{j} +\frac{1}{2} T_{ki}^{j}(0) z^k + \frac{1}{2} \overline{ T_{\ell j}^{i}(0) } \bar z^\ell -\left(\Theta_{k \bar \ell i\bar j}(0) -\frac{1}{4} T_{ki}^{q}(0) \overline{T_{\ell j}^{q}(0)} \right) z^k\bar z^\ell \nonumber \\
&\qquad +\mathrm{(ch.t.)} +o(|z|^2) \,\, , \label{eq:Taylorh_ij} \\
h^{\bar{j}i}(z) &= \delta_{j}^{i} -\frac{1}{2}T_{kj}^{i}(0)z^k -\frac{1}{2}\overline{T_{\ell i}^{j}(0)} \bar{z}^{\ell} +\bigg(\Theta_{k\bar{\ell}j\bar{i}}(0) +\frac{1}{4}T_{kp}^{i}(0)\overline{T_{\ell p}^{j}(0)} \bigg)z^k\bar{z}^{\ell} \nonumber \\
&\qquad +\mathrm{(ch.t.)} +o(|z|^2) \,\, , \label{eq:Taylorh^ij} \\
\det(h_{m\bar{s}}) &= 1 + \frac{1}{2} \theta_k(0) \, z^k
+ \frac{1}{2} \overline{\theta_{\ell}(0)} \, \bar z^\ell - \left(\Theta_{k \bar \ell p \bar p}(0) -\frac{1}{4} \theta_k(0)\overline{\theta_{\ell}(0)} \right)\, z^k \bar z^\ell \nonumber \\
&\qquad +\mathrm{(ch.t.)} + o(|z|^2) \,\, , \label{eq:Taylordeth}
\end{align}
where $\mathrm{(ch.t.)}$ denotes ``charged terms'' of order $2$, \emph{i.e.}, terms of bidegree $(2,0)$ and $(0,2)$.
\end{lemma}

\begin{proof}
A direct computation based on \eqref{eq:Thetaijkl}, \eqref{eq:Tcoord} and \eqref{eq:normcoord} shows that the Taylor expansion of the metric coefficients in holomorphic normal coordinates centred at $p$ is given by
\begin{equation}\label{eq:metric_expansion}
h_{i\bar{j}}(z) = \delta_{i}^{j} + A_{i\bar j}(z) \,\, ,
\end{equation}
where
\begin{align*}
A_{i\bar j}(z) &= \frac{1}{2} T_{ki}^{j}(0) z^k + \frac{1}{2} \overline{ T_{\ell j}^{i}(0) } \bar z^\ell + \left( -\Theta_{k \bar \ell i\bar j}(0)+\frac{1}{4} T_{ki}^{q}(0) \overline{T_{\ell j}^{q}(0)} \right) z^k\bar z^\ell \\
&\qquad +\mathrm{(ch.t.)} + o(|z|^2) \,\, ,
\end{align*}
and so this proves \eqref{eq:Taylorh_ij}. Since $A_{i\bar{j}}(z) = O(|z|)$, \eqref{eq:metric_expansion} implies that
$$
h^{\bar{j}i}(z) = \delta_{j}^{i} - A_{j\bar i}(z) + A_{j\bar p}(z)A_{p\bar i}(z) + o(|z|^2)
$$
and
$$
\det(h_{m\bar{s}})(z) = 1 + A_{i\bar i}(z) + \frac{1}{2} \left( A_{i\bar i}(z)^2 - A_{i\bar j}(z)A_{j\bar i}(z) \right) + o(|z|^2) \,\, .
$$
Therefore, \eqref{eq:Taylorh^ij} and \eqref{eq:Taylordeth} follows by a straightforward computation.
\end{proof}

We now compute the angular average of the volume form, to which we are reduced when integrating radial functions.

\begin{lemma} \label{lem:locholnorm3}
Let $(M^{2n},J,\omega)$ be a Hermitian manifold, and let $\{z^i\}_{i=1,\dots,n}$ be local holomorphic normal coordinates centred at $p \in M$. For every sufficiently small $r>0$, the following Taylor expansions at $p$ hold:
\begin{align}
&\frac{1}{\mathrm{Vol}(S_1^{2n-1})}\int_{S_r^{2n-1}} \det(h_{m\bar{s}}) \,d\sigma_r = r^{2n-1} + \Phi(p) r^{2n+1} + o(r^{2n+1}) \,\, , \label{eq:avdet1} \\
&\frac{1}{\mathrm{Vol}(S_1^{2n-1})} \int_{S_r^{2n-1}} h^{\bar{j}i} z^j \bar{z}^i \det(h_{m\bar s}) \,d\sigma_r = r^{2n+1} + \Psi(p) r^{2n+3} + o(r^{2n+4}) \,\, , \label{eq:avdet2}
\end{align}
where $S^{2n-1}_r \subset \mathbb{C}^n$ is the sphere of radius $r$ centred at the origin, $d\sigma_r$ denotes the area element of $S^{2n-1}_r$ induced by the flat metric of $\mathbb{C}^n$, and
\begin{align}
\Phi &\coloneqq -\frac{1}{2n}\left(\mathrm{scal}^{\mathrm{Ch}} -\frac{1}{4}|\theta|^2\right) \,\, , \label{eq:Phi} \\
\Psi &\coloneqq -\frac{1}{2n(n+1)}\left((n-1)\,\mathrm{scal}^{\mathrm{Ch}} +d^*\theta -\frac{n-2}{4}|\theta|^2 -\frac{1}{4}|T|^2\right) \,\, . \label{eq:Psi}
\end{align}
\end{lemma}

\begin{proof}
A direct computation shows that
\begin{equation} \label{eq:formulasint} \begin{aligned}
&\frac{1}{\mathrm{Vol}(S_1^{2n-1})}\int_{S_r^{2n-1}} d\sigma_r = r^{2n-1} \,\, , \\
&\frac{1}{\mathrm{Vol}(S_1^{2n-1})}\int_{S_r^{2n-1}} z^j\bar{z}^i\,d\sigma_r = \frac{r^{2n+1}}{n}\delta_{i}^{j} \,\, , \\
&\frac{1}{\mathrm{Vol}(S_1^{2n-1})}\int_{S_r^{2n-1}} z^j\bar{z}^iz^k\bar{z}^{\ell}\,d\sigma_r = \frac{r^{2n+3}}{n(n+1)}\big(\delta_{i}^{j}\delta_{\ell}^{k}+\delta_{\ell}^{j}\delta_{i}^{k}\big) \,\, ,
\end{aligned}\end{equation}
while the angular average of charged terms is zero.
Therefore, by \eqref{eq:Taylordeth} and \eqref{eq:formulasint}, we obtain
\begin{align*}
&\frac{1}{\mathrm{Vol}(S_1^{2n-1})}\int_{S_r^{2n-1}} \det(h_{m\bar{s}}) \,d\sigma_r = \\
&\qquad = r^{2n-1} -\frac{1}{n}\left( \Theta_{i\bar i j \bar j}(0)-\frac{1}{4}\theta_i(0)\overline{\theta_i(0)} \right) r^{2n+1} +o(r^{2n+1}) \,\, .
\end{align*}
By \eqref{eq:scalChcoord}, \eqref{eq:|T|coord} and \eqref{eq:Phi}, it follows that
$$
-\frac{1}{n}\left( \Theta_{i\bar i j \bar j}(0)-\frac{1}{4}\theta_i(0)\overline{\theta_i(0)} \right) = \Phi(p)
$$
and so this proves \eqref{eq:avdet1}.

Moreover, by \eqref{eq:Taylorh^ij} and \eqref{eq:Taylordeth} we obtain
\begin{align*}
&h^{\bar{j}i}(z)z^j\bar{z}^i\det(h_{m\bar{s}}(z)) = \\
&=\delta_{j}^{i}z^j\bar{z}^i +\bigg(-\Theta_{k\bar{\ell}p\bar{p}}(0)\delta_{j}^{i} +\frac{1}{4}\theta_k(0)\overline{\theta_{\ell}(0)}\delta_{j}^{i} +\Theta_{k\bar{\ell}j\bar{i}}(0) \\
&\qquad +\frac{1}{4}T_{kp}^{i}(0)\overline{T_{\ell p}^{j}(0)} -\frac{1}{4}T_{kj}^{i}(0)\overline{\theta_{\ell}(0)} -\frac{1}{4}\theta_k(0)\overline{T_{\ell i}^{j}(0)} \bigg)z^j\bar{z}^iz^k\bar{z}^{\ell} \\
&\qquad +\mathrm{(ch.t.)} +o(|z|^4) \,\, ,
\end{align*}
where $\mathrm{(ch.t.)}$ denotes charged terms of order $\leq 4$. Therefore, by \eqref{eq:formulasint}, it follows that
\begin{align*}
&\frac{1}{\mathrm{Vol}(S_1^{2n-1})}\int_{S_r^{2n-1}} h^{\bar{j}i}(z)z^j\bar{z}^i\det(h_{m\bar{s}})\,d\sigma_r = \\
&\qquad = r^{2n+1} +\frac{1}{n(n+1)}\bigg(-n\Theta_{i\bar{i}j\bar{j}}(0) +\Theta_{i\bar{j}j\bar{i}}(0) +\frac{n+2}{4}\theta_i(0)\overline{\theta_i(0)} \\
&\qquad\qquad +\frac{1}{4}T_{ij}^{k}(0)\overline{T_{ij}^{k}(0)}\bigg)r^{2n+3} +o(r^{2n+3}) \,\, .
\end{align*}
By \eqref{eq:scalChcoord}, \eqref{eq:scalChcoord2}, \eqref{eq:scalCh-scalCh}, \eqref{eq:|T|coord} and \eqref{eq:Psi}, we obtain
$$\begin{aligned}
&\frac{1}{n(n+1)} \left(-n\cdot\Theta_{k \bar ki\bar i}(0)+\Theta_{p\bar q q\bar p}(0)+\frac{n+2}{4}\theta_k(0)\theta_{\bar k}(0) +\frac{1}{4} T_{pa}^k(0)\overline{T_{pa}^{k}(0)}\right) = \\
&\qquad = \frac{1}{2n(n+1)}\left(-n\,\mathrm{scal}^{\mathrm{Ch}}(p) +\widetilde{\mathrm{scal}}{}^{\mathrm{Ch}}(p) +\frac{n+2}{4}|\theta(p)|^2 +\frac{1}{4}|T(p)|^2\right) \\
&\qquad = \Psi(p) \,\, ,
\end{aligned}$$
and so this proves \eqref{eq:avdet2}.
\end{proof}

We are now ready to prove Theorem \ref{thm:bubble_test}.

\begin{proof}[Proof of Theorem \ref{thm:bubble_test}]
Let $\{z^i\}_{i=1,\dots,n}$ be local holomorphic normal coordinates centred at $p \in M$. Let $z^i = x^{2i-1} +\sqrt{-1} x^{2i}$ and denote by
$d\mathcal L \coloneqq dx^1 \wedge {\dots} \wedge dx^{2n}$
the standard Lebesgue measure on $\mathbb{C}^n$, so that
\begin{equation} \label{eq:Lebesgue}
\frac{\omega^n}{n!} = 2^n \det(h_{m\bar{s}})\,d\mathcal{L} \,\, .
\end{equation}
Moreover, for brevity, we set
$$
\nu_{2n-1} \coloneqq {\mathrm{Vol}(S_1^{2n-1})} \,\, .
$$

In order to construct suitable test functions for the deformed Yamabe functional, we consider the \emph{standard bubbles}
$$ U_{\varepsilon}(z) = \left(\frac{\varepsilon}{\varepsilon^2+|z|^2}\right)^{n-1} , $$
depending on the parameter $\varepsilon>0$, see, \emph{e.g.}, \cite[p.\ 48]{MR0888880}. They are extremal for the Sobolev inequality in the Euclidean space \cite{MR0448404, MR0463908}.

We consider a ball $B(0,\delta)$ contained in the local chart, and the radial coordinate $r=|z|$. Let $\eta$ be a smooth radial cutoff function with
$$
\eta(r)=1 \quad \text{for $r \leq \tfrac{\delta}{2}$} \,\, , \quad \eta(r)=0 \quad \text{for $r\geq\delta$} \,\, .
$$
Then we consider the radial test functions
$$
u_\varepsilon(z) \coloneqq \eta(|z|)\, U_\varepsilon(z) \,\, .
$$
We now compute the Taylor expansion at $p$ of the quantity
\begin{equation} \label{eq:EHappendix}
\frac{\int_M\left(\psi u_{\varepsilon}^2+k|\nabla u_{\varepsilon}|^2\right)\frac{\omega^n}{n!}}
{\|u_{\varepsilon}\|^2_{L^q}}
\end{equation}
as $\varepsilon \to 0$. For later use, for $k<n$, we set
$$
I_k \coloneqq \int_0^{+\infty} \frac{r^{2n-1+2k}}{(1+r^2)^{2n}} \,dr \,\, , \quad J_k \coloneqq \int_0^{+\infty} \frac{r^{2n-1+2k}}{(1+r^2)^{2n-2}} \,dr
$$
(see, \emph{e.g.}, \cite[Lemma 3.5]{MR0888880}).

By \eqref{eq:avdet1} and \eqref{eq:Lebesgue}, we obtain
\begin{align*}
\|u_\varepsilon\|_{L^q}^q
&=  2^n\int_{B(0,\delta)} \eta(|z|)^q U_\varepsilon(z)^q \det(h_{m\bar s})(z) \, d\mathcal L \\
&= 2^n\nu_{2n-1}\int_{0}^{\frac{\delta}{\varepsilon}} \frac{1}{\varepsilon^{2n}}\eta(r\varepsilon)^q \frac{r^{2n-1}}{\left(1+r^2\right)^{2n}} \cdot \left( (r\varepsilon)^{2n-1} +\Phi(p) (r\varepsilon)^{2n+1}\right. \\
&\qquad \left.+ o(\varepsilon^{2n+1})\right) \, \varepsilon \,dr \\
&= 2^n\nu_{2n-1} \left( \int_{0}^{+\infty} \frac{r^{2n-1}}{\left(1+r^2\right)^{2n}} \, dr + O(\varepsilon^{2n}) \right) \\
&\qquad + 2^n \nu_{2n-1} \Phi(p) \left(\int_{0}^{+\infty} \frac{r^{2n+1}}{\left(1+r^2\right)^{2n}}dr + O(\varepsilon^{2n-2}) \right) \varepsilon^2 + o(\varepsilon^2) \\
&= 2^n\nu_{2n-1}I_0 + 2^n \nu_{2n-1} \Phi(p) I_1\varepsilon^2 + o(\varepsilon^2) \,\, ,
\end{align*}
where $\Phi$ is defined in \eqref{eq:Phi}. By recalling that
\begin{equation} \label{eq:usefulformulas}
\frac{I_1}{I_0}=\frac{n}{n-1} \,\, , \quad \frac{1}{\sigma_{2n}} = 4n(n-1)(\nu_{2n-1}I_0)^{\frac{1}{n}} \,\, ,
\end{equation}
we obtain the following expansion for the denominator of \eqref{eq:EHappendix}:
\begin{equation} \label{eq:expdenom}
\|u_\varepsilon\|_{L^q}^2 = 2^{n+1}\nu_{2n-1}\sigma_{2n}n(n-1)I_0 \left( 1 + \Phi(p) \varepsilon^2 + o(\varepsilon^2) \right) \,\, .
\end{equation}

Furthermore, by \eqref{eq:avdet1}, \eqref{eq:avdet2} and \eqref{eq:Lebesgue}, we obtain
\begin{align*}
&\int_M u_\varepsilon^2 \psi \,\frac{\omega^n}{n!} = \\
&\quad = 2^n\int_{B(0,\delta)} \eta(|z|)^2 U_\varepsilon(z)^2 \psi(z) \det(h_{m\bar s})(z) \, d\mathcal{L} \\
&\quad = 2^n\nu_{2n-1}\varepsilon^{3-2n} \int_{0}^{\frac{\delta}{\varepsilon}} \eta(r\varepsilon)^2 \frac{1}{(1+r^2)^{2n-2}} \left( \psi(p) r^{2n-1} \varepsilon^{2n-1} + o(\varepsilon^{2n-1}) \right)  \,dr \\
&\quad = 2^n\nu_{2n-1}\varepsilon^{3-2n} \left( \psi(p)\varepsilon^{2n-1} \int_{0}^{+\infty} \frac{r^{2n-1}}{(1+r^2)^{2n-2}} dr + O(\varepsilon^{2n-4}) + o(\varepsilon^{2n-1}) \right) \\
&\quad = 2^n\nu_{2n-1}\psi(p) J_0 \,\varepsilon^2 + o(\varepsilon^2)
\end{align*}
and
\begin{align*}
&\|\nabla u_\varepsilon \|_{L^2}^2 = \\
&= 2^{n}\int_{B(0,\delta)} 2h^{\bar\ell k} \left(u_\varepsilon'(z) \frac{\bar z^k}{2|z|}\right) \left(u_\varepsilon'(z) \frac{z^\ell}{2|z|}\right) \det(h_{m\bar s})(z) \, d\mathcal{L} \\
&= 2^{n-1}\nu_{2n-1}\varepsilon^2 \int_{0}^{\frac{\delta}{\varepsilon}} r^{2n-1} \left(\frac{\eta'(r\varepsilon)}{(1+r^2)^{n-1}} - \frac{2(n-1)}{\varepsilon} \frac{r\eta(r\varepsilon)}{(1+r^2)^n} \right)^2 (1+\Psi (p)r^2\varepsilon^2 \\
&\qquad +o(\varepsilon^2)) \, dr \\
&= 2^{n+1}(n-1)^2\nu_{2n-1} \int_{0}^{+\infty} \left(\frac{r^{2n-1}}{(1+r^2)^{2n}}+O(\varepsilon^{2n-2}) \right) (r^2+\Psi(p) r^4\varepsilon^2+o(\varepsilon^2)) \, dr \\
&= 2^{n+1}(n-1)^2\nu_{2n-1} (I_1+\Psi(p) I_2\,\varepsilon^2 + o(\varepsilon^2)) .
\end{align*}
Therefore, using \eqref{eq:usefulformulas} and recalling that
$$
\frac{J_0}{I_0} = \frac{2(2n-1)}{n-2} \,\, , \quad
\frac{I_2}{I_0} = \frac{n(n+1)}{(n-1)(n-2)} \,\, ,
$$
we obtain the following expansion for the numerator of \eqref{eq:EHappendix}:
\begin{multline} \label{eq:expnum}
\int_M (u_{\varepsilon}^2\,\psi + k|\nabla u_{\varepsilon}|^2) \,\frac{\omega^n}{n!} = \\
= 2^n\nu_{2n-1} \left( (2k(n-1)^2I_1) + (\psi(p) J_0+2k(n-1)^2\Psi(p) I_2) \varepsilon^2 +o(\varepsilon^2) \right) \,\, .
\end{multline}
Finally, using \eqref{eq:expdenom} and \eqref{eq:expnum}, we obtain
\begin{align*}
&\frac{\int_M (u_\varepsilon^2\,\psi + k|\nabla u|^2) \,\frac{\omega^n}{n!}}{\|u_\varepsilon\|_{L^q}^2} = \\
&\qquad = \frac{k}{\sigma_{2n}} + \frac{1}{\sigma_{2n}2n(n-2)}\bigg(\frac{2(2n-1)}{n-1}\psi(p) \\
&\qquad\qquad -2nk\big((n-2)\Phi(p)-(n+1)\Psi(p)\big) \bigg) \varepsilon^2 + o(\varepsilon^2) \,\, .
\end{align*}
Moreover, by \eqref{eq:scal-scalCh}, \eqref{eq:Phi} and \eqref{eq:Psi}, we have
$$
2n\big((n-2)\Phi-(n+1)\Psi\big) = \mathrm{scal}(g_{\omega})
$$
and so this proves \eqref{eq:subcrit-appendix}.
\end{proof}

\bibliographystyle{plain}
\bibliography{biblio}

\end{document}